\documentclass[10pt,reqno]{amsart}

\usepackage[a4paper,headheight=0.3cm]{geometry}
\usepackage{fancyhdr}
\usepackage{url}
\usepackage[pdftitle={Probability boxes on totally preordered spaces for multivariate modelling},%
pdfauthor={Matthias C. M. Troffaes and Sebastien Destercke},%
pdfkeywords={lower prevision, p-box, multivariate, Choquet integral, Fréchet bounds, full components}]{hyperref}

\usepackage{mathabx} 
\usepackage{amssymb}
\usepackage{mathtools,amsthm}
\usepackage{enumerate}
\usepackage{hyperref}
\usepackage{tikz}
\usetikzlibrary{snakes,arrows}
\usepackage{xargs}
\usepackage{ifthen}
\usepackage{comment}

\newcommand{\df}{F}
\newcommand{\adf}{G}

\newcommand{\ldf}{\underline{\df}}
\newcommand{\udf}{\overline{\df}}
\newcommand{\aldf}{\underline{\adf}}
\newcommand{\audf}{\overline{\adf}}
\newcommand{\pdomain}{\mathcal{K}}
\newcommand{\plattice}{\mathcal{H}}
\newcommand{\pr}{P}
\newcommand{\apr}{Q}

\newcommand{\nex}{E}
\newcommand{\lpr}{\underline{\pr}}
\newcommand{\upr}{\overline{\pr}}
\newcommand{\alpr}{\underline{\apr}}

\newcommand{\lnex}{\underline{\nex}}
\newcommand{\unex}{\overline{\nex}}
\newcommand{\lpbox}{\lpr_{\ldf,\udf}}
\newcommand{\lpboxlattice}{{\lpr^{\plattice}_{\ldf,\udf}}}

\newcommand{\lnepbox}{\lnex_{\ldf,\udf}}
\newcommand{\unepbox}{\unex_{\ldf,\udf}}
\newcommand{\lcdf}{\lpr_\df}
\newcommand{\lnecdf}{\lnex_\df}
\newcommand{\unecdf}{\unex_\df}
\newcommand{\set}[2]{\left\{#1\colon#2\right\}}
\newcommandx{\gambles}[1][1=]{
  \ifthenelse{\equal{#1}{}}%
  {
    \mathcal{L}
  }{
    \mathcal{L}(#1)
  }
}
\newcommand{\solp}{\mathcal{M}}
\newcommand{\dif}{\,\mathrm{d}}
\newcommand{\card}[1]{{\lvert#1\rvert}}
\newcommand{\pspace}{\Omega}
\newcommand{\linprevs}{\mathcal{P}}
\newcommand{\interior}[1]{\mathrm{int}(#1)}
\newcommand{\closure}[1]{\mathrm{cl}(#1)}
\newcommand{\alpbox}{\lpr_{\aldf,\audf}}

\newcommand{\alnepbox}{\lnex_{\aldf,\audf}}
\newcommand{\aunepbox}{\unex_{\aldf,\audf}}

\newcommand{\SetN}{\mathbb{N}}

\newcommand{\losc}{\underline{osc}}
\newcommand{\uosc}{\overline{osc}}
\newcommand{\reals}{\mathbb{R}}

\theoremstyle{plain}
\newtheorem{theorem}{Theorem}
\newtheorem{proposition}[theorem]{Proposition}
\newtheorem{corollary}[theorem]{Corollary}
\newtheorem{lemma}[theorem]{Lemma}
\theoremstyle{remark}
\newtheorem{example}[theorem]{Example}

\theoremstyle{definition}
\newtheorem{definition}[theorem]{Definition}

\begin{document}
\title[P-boxes on totally preordered spaces for multivariate modelling]{Probability boxes on totally preordered spaces for multivariate modelling}
\author[M. Troffaes]{Matthias Troffaes}
\address{Durham University, Dept. of Mathematical Sciences, Science Laboratories, South Road, Durham
DH1 3LE, United Kingdom} \email{matthias.troffaes@gmail.com}
\author[S. Destercke]{Sebastien Destercke}
\address{CIRAD, UMR1208, 2 place P. Viala, F-34060 Montpellier cedex 1, France} \email{sebastien.destercke@cirad.fr}

\begin{abstract}
A pair of lower and upper cumulative distribution functions, also called probability box or p-box, is among the most popular models used
in imprecise probability theory. They arise naturally in expert elicitation, for instance in cases where bounds are specified on the quantiles of a random variable, or when quantiles are specified only at a finite number of points.
Many practical and formal results
concerning p-boxes already exist in the literature.
In this paper, we provide new efficient tools to construct multivariate p-boxes and develop algorithms to draw inferences from them.
For this purpose, we formalise and extend the theory of p-boxes using Walley's behavioural theory of imprecise probabilities, and heavily rely on its notion of natural extension and existing results about independence modeling.
In particular, we allow p-boxes to be defined on
arbitrary totally preordered spaces, hence thereby also admitting multivariate p-boxes
via probability bounds over any collection of nested sets.
We focus on the cases of independence (using the factorization property), and of unknown dependence (using the Fr\'echet bounds),
and we show that our approach extends the probabilistic arithmetic of Williamson and Downs.
Two design problems---a damped oscillator, and a river dike---demonstrate the practical feasibility of our results.
\end{abstract}
\keywords{Lower prevision, p-box, multivariate, Choquet integral, Fr\'echet bounds, full component}
\maketitle
\thispagestyle{fancy}

\section{Introduction}

Imprecise probability \cite{walley1991} refers to
uncertainty models applicable in situations where the available
information does not allow us to single out a unique probability
measure for all random variables involved. Examples of such models
include $2$- and $n$-monotone capacities \cite{choquet1953}, lower
and upper previsions \cite{williams1975,williams2007,walley1991},
belief functions \cite{shafer1976}, credal sets \cite{levi1980a},
possibility and necessity measures \cite{dubois1988,cooman1997ab},
interval probabilities \cite{weichselberger2001}, and coherent risk
measures \cite{artzner1999,delbaen2002}.

Unlike classical probability models, which are described by
probability measures, imprecise probability models
require more complex mathematical tools, such as non-linear
functionals and non-additive measures \cite{walley1991}. It is therefore of
interest to consider particular imprecise probability models that
yield simple mathematical descriptions, possibly at the expense of
generality, but gaining ease of use, elicitation, and graphical
representation.

One such model is considered in this paper: pairs of lower and upper
distribution functions, also called \emph{probability boxes}, or
briefly, p-boxes \cite{ferson2003,2006:ferson}. P-boxes are often
used in risk or safety studies, in which cumulative distributions
play an essential role. Many theoretical properties and practical aspects of p-boxes have already been studied in the literature. Previous work includes probabilistic
arithmetic \cite{WilliamsonDowns1990}, which provides a very
efficient numerical framework for particular statistical inferences
with p-boxes (and which we generalise in this paper). In \cite{2008:ferson}, p-boxes are connected to
info-gap theory \cite{2006:benhaim}. The relation
between p-boxes and random sets was investigated in
\cite{KrieglerHeld2005} and applied in \cite{2008:oberguggenberger}:
many results and techniques applicable to random sets are also
applicable to p-boxes. Finally, a recent extension of p-boxes to
arbitrary finite spaces \cite{desterckedubois2008} yields potential
application of p-boxes to a much more general set of problems, such as robust design analysis \cite{desterckedubois2008b,2008:fuchs}, and signal processing \cite{2009:desterckeb}. 

In this paper, we study p-boxes within the framework of
the theory of coherent lower previsions. Coherent lower previsions
were introduced by Williams \cite{williams1975} as a generalisation of de Finetti's
work \cite{finetti19745}, and were developed further by
Walley \cite{walley1991}. Coherent lower previsions generalize
many of the other imprecise probability models in the
literature, and are equivalent to closed convex sets of finitely
additive probability measures.

Studying p-boxes by means of lower previsions has at least two advantages:
\begin{itemize}
\item Lower previsions can be defined on arbitrary spaces, and thus enable p-boxes to be used for more general problems, and not just problems concerning the real line.
\item Lower previsions come with a powerful inference tool, called \emph{natural extension}, which reflects the least-committal consequences of any given assessments. Natural extension generalises many known extensions, including for instance Choquet integration for $2$-monotone measures. In this paper, we study the natural extension of a p-box, and we derive a number of useful expressions for it. This leads to new numerical tools that provide exact inferences on arbitrary events, and even on arbitrary (bounded) random quantities.
\end{itemize}

From the point of view of coherent lower previsions, p-boxes have
already been studied briefly in \cite[Section~4.6.6]{walley1991} and
\cite{troffaes2005}. Lower and upper
distribution functions associated with a sequence of moments have also been considered \cite{miranda2006c}.

As already mentioned, \cite{desterckedubois2008} extended
p-boxes to finite totally preordered spaces. In this paper, we
extend p-boxes further to arbitrary totally preordered spaces.
Our generalisation has many useful features that classical p-boxes do not have:
\begin{itemize}
\item We encompass, in one sweep, p-boxes defined on finite spaces, as well as (continuous) p-boxes on closed real
intervals.
\item Perhaps even more importantly, as we do not impose
anti-symmetry on the ordering, we can also handle product spaces by
considering an appropriate total preorder---for instance, one
induced by a metric---and thus also admit multivariate non-finite p-boxes,
which have not been considered before.
Whence, we can specify p-boxes directly on the product space. Contrast this with the usual multivariate approach to p-boxes, such as probabilistic arithmetic \cite{WilliamsonDowns1990}, that consider one marginal p-box per dimension and draw inferences from a joint model built around some information about variable dependencies (of course, we can still do the same, and will do so in Section~\ref{sec:marginals}).
\item  Our approach is also useful in elicitation, as it allows uncertainty to be expressed as probability bounds over any collection of (possibly multivariate) nested sets---also see \cite{2008:fuchs,2009:destercke} for a discussion of similarly constructed models. Thus, unlike classical p-boxes, we are not restricted to events of the type $[-\infty,x]$, even on the real line.
\end{itemize}

Our approach is thus rather different, and far more general, than the one usually considered for inferences with p-boxes. Indeed, we first define a joint p-box over some multivariate space of interest, either directly or by using marginal models and a dependence model, after which we draw exact inferences from this joint p-box using natural extension. In contrast, usual methods~\cite{2004:regan} such as for instance probabilistic arithmetic \cite{WilliamsonDowns1990} start out with marginal p-boxes each defined on the real line, and provide tools to make inferences for specific multivariate events.

The paper is organised as follows: Section~\ref{sec:prel} provides a
brief introduction to the theory of coherent lower previsions, used
in the rest of the paper. Section~\ref{sec:charac} then introduces
and studies the p-box model from the point of view of
lower previsions. Section~\ref{sec:nex} provides an expression
for the natural extension of a p-box to all events, via the partition topology induced by the equivalence classes of the preorder and additivity on full components.
Section~\ref{sec:nex-gambles} studies the natural extension to all gambles, via the Choquet integral. Section~\ref{sec:induced-preorder} studies an important special case of p-boxes whose preorder is induced by a real-valued mapping, as this will usually be the most convenient way to specify a multivariate p-box.
Section~\ref{sec:marginals} discusses the construction of such multivariate p-boxes from marginal coherent lower previsions under arbitrary dependency models. For two important special cases---epistemic independence, and completely unknown dependence (the Fr\'echet case)---closed expressions are derived.
We also derive probabilistic arithmetic as a special case of our approach.
Section~\ref{sec:example:engineering} demonstrates the theory with some examples. As a first example, we infer bounds on the expectation of the damping ratio of a damped harmonic oscillator whose parameters are described by a multivariate p-box. As a second example, we derive the expected overflow height for a river dike, again using a multivariate p-box.
Finally, Section~\ref{sec:conclusions} ends with our main conclusions
and open problems.

\section{Preliminaries}
\label{sec:prel}

We start with a brief introduction to the imprecise probability models
that we shall use in this paper. We refer to \cite{1854:boole}, \cite{williams1975} and
\cite{walley1991} for more details. See also \cite{miranda2008} for
a brief summary of the theory.

Let $\pspace$ be the possibility space. A subset of $\pspace$ is
called an \emph{event}. A \emph{gamble} on $\pspace$ is a bounded
real-valued function on $\pspace$. The set of all gambles on
$\pspace$ is denoted by $\gambles[\pspace]$, or simply by $\gambles$
if the possibility space is clear from the context. A particular
type of gamble is the \emph{indicator} of an event $A$, which is the
gamble that takes the value $1$ on elements of $A$ and the value $0$
elsewhere, and is denoted by $I_A$, or simply by $A$ if no confusion
with $A$ as event is possible.

A \emph{lower prevision} $\lpr$ is a real-valued functional defined
on an arbitrary subset $\pdomain$ of $\gambles$, and is interpreted
as follows: for any gamble $f$ in $\pdomain$ and any $\epsilon>0$,
the transaction $f-\lpr(f)+\epsilon$ is acceptable to the subject who has assessed this lower prevision.
Hence, lower previsions summarize a subject's supremum buying prices
for a collection of gambles, and it can be argued that in this way
they model a subject's belief about the true state $x$ in $\pspace$
(see \cite{walley1991} for a more in-depth explanation). A lower
prevision defined on a set of indicators of events is usually called
a \emph{lower probability}.

By $\upr$, we denote the conjugate \emph{upper prevision} of $\lpr$:
for every $f$ such that $-f\in\pdomain$, $\upr(f)=-\lpr(-f)$. The
upper prevision $\upr(f)$ can be interpreted as a subject's infimum
selling price for $f$.

A lower prevision on $\pdomain$ is called \emph{coherent} (see
\cite[p.~5]{williams1975} and \cite[pp.~73--75,
Sec.~2.5]{walley1991}) when for all $p$ in $\SetN$, all $f_0$,
$f_1$, \dots, $f_p$ in $\pdomain$ and all non-negative real numbers
$\lambda_0$, $\lambda_1$, \dots, $\lambda_p$,
\begin{equation*}
  \sup_{x \in \pspace}
  \left[
    \sum_{i=1}^{p} \lambda_i
    (f_i(x)-\lpr(f_i))-\lambda_0(f_0(x)-\lpr(f_0))
  \right]
  \geq 0.
\end{equation*}
A lower prevision on the set $\gambles$ of all gambles is coherent
if and only if (see \cite[p.~11, Sec.~1.2.2]{williams1975} and
\cite[p.~75, Sec.~2.5.5]{walley1991})
\begin{itemize}
 \item[(C1)] $\lpr(f)\geq \inf f$,
 \item[(C2)] $\lpr(\lambda f)=\lambda \lpr(f)$, and
 \item[(C3)] $\lpr(f+g)\geq \lpr(f)+\lpr(g)$
\end{itemize}
for all gambles $f$, $g$ and all non-negative real numbers $\lambda$.

A functional on $\gambles$ satisfying $\pr(f)\geq \inf f$ and
$\pr(f+g)=\pr(f)+\pr(g)$ for any pair of gambles $f$ and $g$ is
called a \emph{linear prevision} on $\gambles$ \cite[p.~88,
Sec.~2.4.8]{walley1991}, and the set of all linear previsions on
$\gambles$ is denoted by $\linprevs$. A linear prevision is the
expectation operator with respect to its restriction to events,
which is a finitely additive probability.

An alternative characterisation of coherence via linear previsions
goes as follows. Let $\lpr$ be a lower prevision on $\pdomain$ and
let $\solp(\lpr)$ denote the set of all linear previsions on
$\gambles$ that dominate $\lpr$ on $\pdomain$:
\begin{equation*}
  \solp(\lpr)=\{\apr\in\linprevs\colon (\forall f\in\pdomain)(\apr(f)\ge\lpr(f))\}
\end{equation*}
Then $\lpr$ is coherent if and only if $\lpr$ agrees with the lower
envelope of $\solp(\lpr)$  on $\pdomain$, that is, if and only if
$\lpr(f)=\min_{\apr\in\solp(\lpr)}\apr(f)$ for all $f\in\pdomain$
(see \cite[p.~18]{williams1975} and \cite[p.~138,
Sec.~3.3.3]{walley1991}). A consequence of this is that a lower
envelope of coherent lower previsions is again a coherent lower
prevision.

Given a coherent lower prevision $\lpr$ on $\pdomain$, its
\emph{natural extension} \cite[Chapter~3]{walley1991} to a larger
set $\pdomain_1$ of gambles ($\pdomain_1 \supseteq \pdomain$) is the pointwise smallest
coherent lower prevision on $\pdomain_1$ that agrees with $\lpr$ on
$\pdomain$. Because it is the pointwise smallest, the natural
extension is the most conservative (or least-committal) coherent
extension, and thereby reflects the minimal behavioural consequences
of $\lpr$ on $\pdomain_1$.

Taking natural extension is transitive \cite[p.~98,
Cor.~4.9]{troffaes2005}: if $\lnex_1$ is the natural extension of
$\lpr$ to $\pdomain_1$ and $\lnex_2$ is the natural extension of
$\lnex_1$ to $\pdomain_2\supseteq \pdomain_1$, then $\lnex_2$ is
also the natural extension of $\lpr$ to $\pdomain_2$. Hence, if we
know the natural extension of $\lpr$ to the set of all gambles then
we also know the natural extension of $\lpr$ to any set of gambles
that includes $\pdomain$. The natural extension to all gambles is
usually denoted by $\lnex$. It holds that
$\lnex(f)=\min_{\apr\in\solp(\lpr)}\apr(f)$ for any $f\in\gambles$
\cite[p.~136, Sec.~3.4.1]{walley1991}.

A particular class of coherent lower previsions of interest in this
paper are \emph{completely monotone lower previsions}
\cite{cooman2008,cooman2008c}. A lower prevision $\lpr$ defined on
a lattice of gambles $\pdomain$, i.e., a set of gambles closed under
point-wise maximum and point-wise minimum, is called $n$-monotone if
for all $p\in\SetN$, $p\leq n$, and all $f$, $f_1$, \dots, $f_p$ in
$\pdomain$:
  \begin{equation*}
    \sum_{I\subseteq\{1,\dots,p\}}(-1)^\card{I}
    \lpr\left(f\wedge\bigwedge_{i\in I}f_i\right)\geq0.
  \end{equation*}
A lower prevision which is $n$-monotone for all $n\in\SetN$ is
called \emph{completely monotone}.

\section{P-Boxes}
\label{sec:charac}

In this section, we introduce the formalism of p-boxes
defined on totally preordered (not necessarily finite) spaces, as well
as the field of events $\plattice$, which will be instrumental to
study the natural extension of p-boxes. Hence, in contrast to
the work by \cite{ferson2003}, we do not restrict p-boxes to
intervals on the real line.

Let $(\pspace,\preceq)$ be a total preorder: so $\preceq$ is transitive and
reflexive and any two elements are comparable. We write $x\prec y$ for
$x\preceq y$ and $x\not\succeq y$, $x\succ y$ for $y\prec x$, and $x\simeq y$ for $x\preceq y$ and $y\preceq x$. For any two $x$,
$y\in\pspace$ exactly one of $x\prec y$, $x\simeq y$, or $x\succ y$ holds. We also use the following common notation for intervals in $\pspace$:
\begin{align*}
  [x,y]&=\{z\in\pspace\colon x\preceq z\preceq y\} \\
  (x,y)&=\{z\in\pspace\colon x\prec z\prec y\}
\end{align*}
and similarly for $[x,y)$ and $(x,y]$.

For simplicity,
we assume that $\pspace$ has a smallest element $0_\pspace$ and a
largest element $1_\pspace$. This is no essential assumption, since
we can always add these two elements to the space $\pspace$.

A \emph{cumulative distribution function} is a mapping $\df:\pspace\rightarrow [0,1]$ which is
non-decreasing and satisfies moreover $\df(1_\pspace)=1$. For each $x\in\pspace$, we interpret $\df(x)$ as
the probability of the interval $[0_\pspace,x]$. Note that we do not impose
$\df(0_\pspace)=0$, so we allow $\{0_\pspace\}$ to carry non-zero mass, which happens commonly if $\pspace$ is finite. Also note that distribution functions are not
assumed to be right-continuous---this would make no sense since we have no topology defined yet on $\pspace$---but even if there is a topology for $\pspace$, we make no continuity assumptions.

By $\pspace/\simeq$ we denote the quotient set of $\pspace$ with respect to the equivalence relation $\simeq$ induced by $\preceq$, that is:
\begin{align*}
  [x]_{\simeq}&=\{y\in\pspace\colon y\simeq x\} \text{ for any }x\in\pspace \\
  \pspace/\simeq&=\{[x]_{\simeq}\colon x\in\pspace\}
\end{align*}
Because $\df$ is non-decreasing, $\df$ is constant on elements
$[x]_{\simeq}$ of $\pspace/\simeq$.

\begin{definition}
  A \emph{probability box}, or \emph{p-box}, is a pair $(\ldf,\udf)$ of
  cumulative distribution functions from $\pspace$ to $[0,1]$ satisfying $\ldf\leq\udf$.
\end{definition}

Note that some definitions, such as in \cite{desterckedubois2008}, differ in that they use the coarsest preorder for which $\ldf$ and $\udf$ are non-decreasing, effectively imposing that $x \simeq y$ if and only if $\ldf(x)=\ldf(y)$ and $\udf(x)=\udf(y)$. This paper follows \cite{desterckedubois2006}, only imposing that $\ldf$ and $\udf$ are non-decreasing, so $x \simeq y$ implies $\ldf(x)=\ldf(y)$ and $\udf(x)=\udf(y)$, but not the other way around, thereby admitting more preorders, possibly leading to tighter bounds (see Example~\ref{ex:ordering-importance} further on).

A p-box is interpreted as a lower and an upper
cumulative distribution function. In Walley's framework, this
means that a p-box is interpreted as a
lower probability $\lpbox$ on the set of events
\begin{equation*}
  \pdomain=\set{[0_\pspace,x]}{x\in\pspace}\cup\set{(y,1_\pspace]}{y\in\pspace}
\end{equation*}
by
\begin{equation*}
  \lpbox([0_\pspace,x])=\ldf(x)\text{ and }\lpbox((y,1_\pspace])=1-\udf(y).
\end{equation*}
In the particular case of p-boxes on $[a,b]\subseteq\reals$ it was mentioned by
\cite[Section~4.6.6]{walley1991} and proven by \cite[p.~93,
Thm.~3.59]{troffaes2005} that $\lpbox$ is coherent. More generally, it is
straightforward to show that p-boxes on an arbitrary totally preordered space $(\pspace,\preceq)$ are coherent as
well.\footnote{The proof is virtually identical to the one given in \cite[p.~93, Thm.~3.59]{troffaes2005}.}
We denote by $\lnepbox$ the natural extension of
$\lpbox$ to all gambles. We study this natural extension extensively
further on.

When $\ldf=\udf$, we say
that $(\ldf,\udf)$ is \emph{precise}, and we denote the corresponding lower
prevision on $\pdomain$ by $\lcdf$ and its natural extension to
$\gambles$ by $\lnecdf$ (with $\df:=\ldf=\udf$).

A few examples of p-boxes on $[0,1]$ are illustrated in Fig~\ref{fig:genpbox}.

\begin{figure}
\begin{tikzpicture}[scale=0.85]
\tikzstyle{indf}=[circle,draw=black,fill=black,scale=0.3]
\tikzstyle{notin}=[circle,draw=black,solid,scale=0.3]
\tikzstyle{inout}=[circle,draw=gray,fill=gray,scale=0.3]
\draw[->] (-0.1,0) -- (5.5,0) node[below right] {$\pspace$}; \draw
(5,0.1) -- (5,-0.1) node[below] {$1$}; \draw[->] (0,-0.1)
node[below] {$0$} -- (0,3) node[left] {1} -- (0,3.2);
\draw[domain=0:5,thick]  plot[id=cdf1,smooth] function{3*norm((x-2)*2)};
\draw[domain=0:5,thick]  plot[id=cdf2,smooth] function{3*norm((x-3)*2)};
\node at (1.7,1.5) {$\udf$}; \node at (3.3,1.5) {$\ldf$};

\begin{scope}[xshift=8cm]
\draw[->] (-0.1,0) -- (5.5,0) node[below right] {$\pspace$};
\draw (5,0.1) -- (5,-0.1) node[below] {$1$};
\draw[->] (0,-0.1) node[below] {$0$} -- (0,3.2) node[left] {1};
\draw (0,0.6) node[indf] {} --  (1,0.6) ;
\draw[dashed] (1,0.6) node[notin] {} -- (1,0.9);
\draw (1,0.9) node[indf] {} -- (3,0.9)    ;
\draw[dashed] (3,0.9) node[indf] {} -- (3,2.4) node[midway,right] {$\ldf$} ;
\draw (3,2.4) node[notin] {} -- (4.5,2.4) ;
\draw[dashed]  (4.5,2.4) node[notin] {} -- (4.5,3) ;
\draw (4.5,3) node[indf] {} -- (5,3);

\draw  (0,0.9) node[indf] {}  --  (0.5,0.9) ;
\draw[dashed] (0.5,0.9) node[notin] {} -- (0.5,1.8);
\draw (0.5,1.8) node[indf] {} -- (2,1.8)    ;
\draw[dashed] (2,1.8) node[notin] {} -- (2,2.4) node[indf]  {} -- (2,3) node[notin] {} ;
\draw (2,3) -- (5,3);
\draw (2,-0.1) node[below] {$x$} -- (2,0.1) ;
\node[above] at (2,3) {$\udf$};

\end{scope}

\begin{scope}[yshift=-5cm]
\draw[->] (-0.1,0) -- (5.5,0) node[below right] {$\pspace$};
\draw (5,0.1) -- (5,-0.1) node[below] {$1$};
\draw[->] (0,-0.1) node[below] {$0$} -- (0,3) node[left] {1} -- (0,3.2);
\draw[domain=0:5, thick]  plot[id=cdf3,smooth] function{3*norm((x-2.5)*2)};
\node[right] at (2.5,1.5) {$\df:=\udf=\ldf$};

\end{scope}

\begin{scope}[yshift=-5cm,xshift=8cm]
\draw[->] (-0.1,0) -- (5.5,0) node[below right] {$\pspace$};
\draw (5,0.1) -- (5,-0.1) node[below] {$1$};
\draw[->] (0,-0.1) node[below] {$0$} -- (0,3) node[left] {1} -- (0,3.2);
\draw[thick] (0,0.3) node [indf] {} .. controls (1.8,0.3) .. (2.5,1.2) ;
\draw[dashed] (2.5,1.2) node[notin] {} --  (2.5,1.6) node[indf] {} -- (2.5,2) node[notin] {} ;
\draw[thick] (2.5,2) .. controls (3,3) .. (5,3) node [indf] {};
\node[right] at (2.5,1.6) {$\ldf$} ;
\draw[thick]  (0,0.6) node [indf] {} .. controls (1.5,0.6) and (1,3)  .. (4,3) node[midway,above] {$\udf$}  -- (5,3) ;

\end{scope}
\end{tikzpicture}
\caption{Examples of p-boxes on $[0,1]$.}
\label{fig:genpbox}
\end{figure}
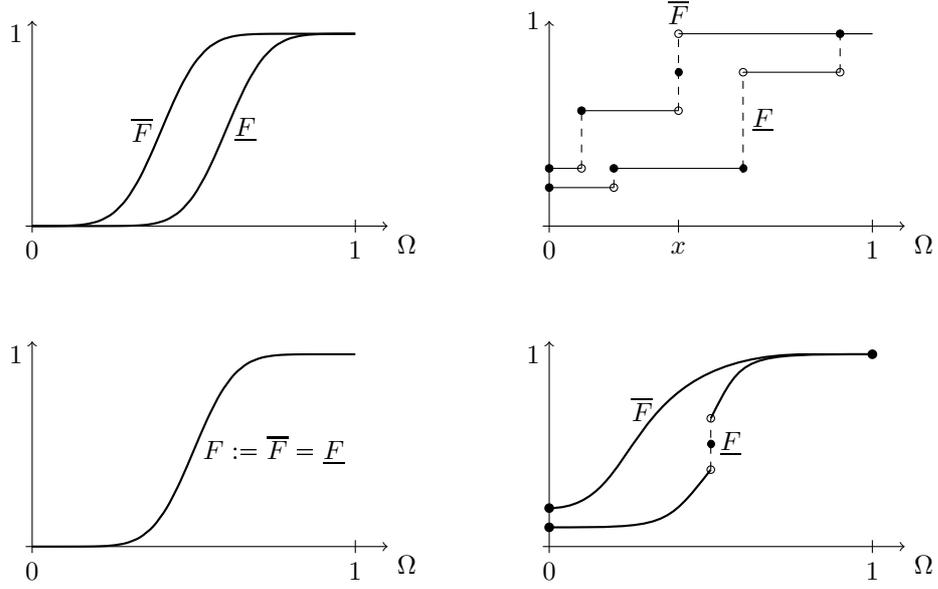

Given a p-box, we can consider the set of distribution
functions that lie between $\ldf$ and $\udf$,
\begin{equation*}
  \Phi(\ldf,\udf)=\set{\df}{\ldf\leq\df\leq\udf}.
\end{equation*}
We can easily express the natural extension $\lnex_{\ldf,\udf}$ of
$\lpbox$ in terms of $\Phi(\ldf,\udf)$: $\lnex_{\ldf,\udf}$ is the
lower envelope of the natural extensions of the cumulative distribution functions\footnote{The natural extension of a cumulative distribution function $\df$ is simply understood to be the natural extension of the precise p-box $(\df,\df)$.} $\df$ that lie
between $\ldf$ and $\udf$:
\begin{equation}\label{eq:lower-envelope-df}
  \lnepbox(f)=\inf_{\df\in\Phi(\ldf,\udf)}\lnecdf(f)
\end{equation}
for all gambles $f$ on $\pspace$. A similar result for p-boxes on
$[0,1]$ can be found in \cite[Section~4.6.6]{walley1991} and
\cite[Theorem~2]{miranda2006c}. To see why this holds, note that
$\lpbox$ is the lower envelope of the set $\solp(\lpbox)$ of linear
previsions dominating $\lpbox$, because $\lpbox$ is a coherent lower
probability. Each of the linear previsions $\apr$ in $\solp(\lpbox)$
has a cumulative distribution function $\df_{\apr}$, and it is easy to see that
$\df_{\apr}\in\Phi(\ldf,\udf)$. Conversely, any linear prevision
whose cumulative distribution function $\df$ belongs to $\Phi(\ldf,\udf)$ must belong to
$\solp(\lpbox)$. Therefore, the set $\solp(\lpbox)$ coincides with
the set of all linear previsions whose cumulative distribution function belongs to
$\Phi(\ldf,\udf)$, which establishes
Eq.~\eqref{eq:lower-envelope-df}. We shall use this equation
repeatedly in subsequent proofs. This allows us moreover to give a
sensitivity analysis interpretation to p-boxes: we can
always regard them as a model for the imprecise knowledge of a
cumulative distribution function.

We end this section with a trivial, yet very useful, approximation theorem:

\begin{theorem}\label{thm:pbox-approximation}
  Let $\lpr$ be any coherent lower prevision defined on
  $\gambles$. The least conservative p-box $(\ldf,\udf)$
  on $(\pspace,\preceq)$ whose natural extension is dominated by
  $\lpr$ is given by
  \begin{align*}
    \ldf(x)&=\lpr([0_\pspace,x])) & \udf(x)&=\upr([0_\pspace,x])
  \end{align*}
  for all $x\in\pspace$.
\end{theorem}
\begin{proof}
  Obviously, the natural extension $\lnepbox$ of the p-box, with
  $\ldf$ and $\udf$ as above, is dominated by $\lpr$, because $\lpr$
  is an extension of $\lpbox$ (see for instance \cite[p.~98,
  Prop.~4.7]{troffaes2005}).

  Any other p-box $(\aldf,\audf)$ whose natural extension is dominated
  by $\lpr$ must satisfy:
  \begin{align*}
    \aldf(x)&=\alpbox([0_\pspace,x])=\alnepbox([0_\pspace,x])\le\lpr([0_\pspace,x])=\ldf(x)
    \\
    \audf(x)&=1-\alpbox((x,1])=\aunepbox([0_\pspace,x])\ge\upr([0_\pspace,x])=\udf(x)
  \end{align*}
  so $(\ldf,\udf)$ is indeed the least conservative one.
\end{proof}

\section{Natural Extension to All Events}
\label{sec:nex}

The remainder of this paper is devoted to the natural
extension $\lnepbox$ of $\lpbox$, and to various convenient
expressions for it. We start by giving the form of the
natural extension on the field of events generated by $\pdomain$.

\subsection{Extension to the Field Generated by the Domain}

Let $\plattice$ be the field of events generated by the domain
$\pdomain$ of the p-box, i.e., events of the type
\begin{equation*}
  [0_\pspace,x_1]\cup(x_2,x_3]\cup\dots\cup(x_{2n},x_{2n+1}]
\end{equation*}
for $x_1\prec x_2\prec x_3\prec \dots\prec x_{2n+1}$ in $\pspace$ (if $n$ is $0$
we simply take this expression to be $[0_\pspace,x_1]$)
and
\begin{equation*}
  (x_2,x_3]\cup\dots\cup(x_{2n},x_{2n+1}]
\end{equation*}
for $x_2\prec x_3\prec \dots\prec x_{2n+1}$ in $\pspace$. Clearly, these events
form a field: the union and intersection of any two events in
$\plattice$ is again in $\plattice$, and the complement of any event
in $\plattice$ also is again in $\plattice$.

To simplify the description of this field, and the expression of
natural extension, we introduce an element $0_\pspace-$ such
that:\footnote{The cunning reader notes that one could avoid
  introducing $0_\pspace-$ by imposing
  $\ldf(0_\pspace)=\udf(0_\pspace)=0$. However, this leads to an
  apparent loss of generality when linking p-boxes to other
  uncertainty models, and a slightly more complicated equation for the
  natural extension---therefore, for our purpose, it is simpler to
  stick to the most general formulation.
  ($\ldf(1_\pspace)=\udf(1_\pspace)=1$ follows from coherence, so
  nothing is lost there.)}
\begin{gather*}
  0_\pspace-\prec x\text{ for all }x\in\pspace
  \\
  \df(0_\pspace-)=\ldf(0_\pspace-)=\udf(0_\pspace-)=0
\end{gather*}
In particular, $(0_\pspace-,x]=[0_\pspace,x]$.
If we let $\pspace^*=\pspace\cup\{0_\pspace-\}$, then
\begin{equation}\label{eq:plattice}
  \plattice=\{
  (x_0,x_1]\cup(x_2,x_3]\cup\dots\cup(x_{2n},x_{2n+1}]
  \colon
  x_0\prec x_1\prec \dots\prec x_{2n+1}\text{ in }\pspace^*\}.
\end{equation}

Since the procedure of natural extension is transitive, in order to
calculate the natural extension of $\lpbox$ to all gambles, we shall
first consider the extension from $\pdomain$ to $\plattice$, then
the natural extension from $\plattice$ to the set of all events,
and finally the natural extension from the set of all events to the
set of all gambles.

In the case of a precise p-box, $\lcdf$ has a unique extension to a
finitely additive probability measure on $\plattice$.

\begin{proposition}\label{prop:nex-lattice-precise}
  $\lnecdf$ restricted to $\plattice$ is a finitely additive probability measure. Moreover,
  for any $A\in\plattice$, that is
  $A=(x_0,x_1]\cup(x_2,x_3]\cup\dots\cup(x_{2n},x_{2n+1}]$ with $x_0\prec x_1\prec \dots\prec x_{2n+1}$ in $\pspace^*$, it holds that
  \begin{equation}\label{eq:nex-lattice-precise}
    \lnecdf(A)=\sum_{k=0}^{n}\left(\df(x_{2k+1})-\df(x_{2k})\right)
  \end{equation}
\end{proposition}
\begin{proof}
  Because
  \begin{align*}
    \lnecdf([0_\pspace,x])&=\lcdf([0_\pspace,x])=\df(x) \text{ and}\\
    \unecdf([0_\pspace,x])&=1-\lnecdf((x,1_\pspace])=1-\lcdf((x,1_\pspace])=\df(x)
  \end{align*}
  for all $x$, $\lnecdf$ is linear on the linear space spanned by
  $\{[0_\pspace,x]\colon x\in\pspace\}$ \cite[p.~102, Prop.~4.18]{troffaes2005}. This linear space includes $\plattice$,
  which proves that $\lnecdf$ is additive on $\plattice$, and consequently
  it is a finitely additive probability measure on $\plattice$.
  The expressions for $\lnecdf(A)$ follow immediately.
\end{proof}

We now extend Proposition~\ref{prop:nex-lattice-precise} to
p-boxes.
\begin{proposition}\label{prop:nex-lattice}
  For any $A\in\plattice$, that is
  $A=(x_0,x_1]\cup(x_2,x_3]\cup\dots\cup(x_{2n},x_{2n+1}]$ with $x_0\prec x_1\prec \dots\prec x_{2n+1}$ in $\pspace^*$,
  it holds that $\lnepbox(A)=\lpboxlattice(A)$, where
  \begin{equation}
    \label{eq:nex-lattice}
    \lpboxlattice(A)=
    \sum_{k=0}^{n}\max\{0,\ldf(x_{2k+1})-\udf(x_{2k})\}.
  \end{equation}
\end{proposition}
\begin{proof}
  We first show that $\lnepbox(A)\le\lpboxlattice(A)$. Consider a cumulative distribution function $\df$ in $\Phi(\ldf,\udf)$ which satisfies
  \begin{align*}
    \df(x_{2k})&=\udf(x_{2k})
    \\
    \df(x_{2k+1})&=\max\{\ldf(x_{2k+1}),\udf(x_{2k})\}
  \end{align*}
  for all $k=0,\dots,n$. Note that
  $0\le\df(x_0)\le\dots\le\df(x_{2n+1})\le 1$, so there is
  a cumulative distribution function satisfying the above equalities.  Secondly, note that for each $k=0,\dots,2n+1$, $\df$ satisfies
  $\ldf(x_k)\le\df(x_k)\le\udf(x_k)$.  Hence, there is indeed a cumulative distribution function $\df$ in
  $\Phi(\ldf,\udf)$ satisfying the above equalities.
  By Eq.~\eqref{eq:nex-lattice-precise},
  \begin{align*}
    \lnecdf(A)
    &=
    \sum_{k=0}^{n}\max\{\ldf(x_{2k+1}),\udf(x_{2k})\} - \udf(x_{2k})
    \\
    &=
    \sum_{k=0}^{n}\max\{0,\ldf(x_{2k+1})-\udf(x_{2k})\}
    =
    \lpboxlattice(A)
  \end{align*}
  with respect to $\df$. Using Eq.~\eqref{eq:lower-envelope-df}, we deduce that $\lpboxlattice(A)\ge\lnepbox(A)$.

  Next, we show that $\lnepbox(A)\ge\lpboxlattice(A)$.
  Let $\df$ be any cumulative distribution function in $\Phi(\ldf,\udf)$. Then,
  \begin{equation*}
  \lnecdf((x_{2k},x_{2k+1}])=\df(x_{2k+1})-\df(x_{2k})\ge\max\{0,\ldf(x_{2k+1})-\udf(x_{2k})\}
  \end{equation*}
  since
  $\df(x_{2k+1})\ge\ldf(x_{2k+1})$ and
  $-\df(x_{2k})\ge-\udf(x_{2k})$. But, because $\lnecdf$ is a finitely additive probability measure on $\plattice$
  (Proposition~\ref{prop:nex-lattice-precise}), $\lpboxlattice(A)\le\lnecdf(A)$. This holds for any $\df$ in
  $\Phi(\ldf,\udf)$, and hence Eq.~\eqref{eq:lower-envelope-df} implies that $\lpboxlattice(A)\le\lnepbox(A)$.
\end{proof}

Note that it is possible to derive a closed expression for the upper
prevision $\unepbox$ as well, similar to Eq.~\eqref{eq:nex-lattice},
however that expression is not as easy to work with. In practice, to
calculate $\unepbox(A)$ for an event $A\in\plattice$, it is by far
easiest first to calculate $\lnepbox(A^c)$ using
Eq.~\eqref{eq:nex-lattice} (observe that also $A^c\in\plattice$), and
then to apply the conjugacy relation:
\begin{equation*}
  \unepbox(A)=1-\lnepbox(A^c).
\end{equation*}

 The lower probability $\lpr_{\ldf,\udf}$ on
 $\pdomain$ determined by the p-box usually does not have a \emph{unique} coherent extension to the field
 $\plattice$ (unless $\ldf=\udf$),
as shown in the following example.

\begin{example}\label{ex:nex-lattice-not-unique}
 Consider the distribution functions $\ldf$ and $\udf$ given by $\udf(x)=x$ for $x \in [0,1]$ and
 $\ldf(x)=0$ if $x\leq 0.5$, $\ldf(x)=2(x-0.5)$ if $x \geq 0.5$.

 From Proposition~\ref{prop:nex-lattice-precise}, both $\lpr_{\ldf}$
 and $\lpr_{\udf}$ have a unique extension to the field $\plattice$.
 Let us define $\lpr_1$ on $\plattice$ by
 $\lpr_1(A):=\min\{\lpr_{\ldf}(A),\lpr_{\udf}(A)\}$ for all $A$.
 Then $\lpr_1$ is a coherent lower prevision on $\plattice$, and it is
 not difficult to show that $\lpr_1=\lpr_{\ldf,\udf}(A)$ for any $A \in \pdomain$.

 Now, given the interval $(0.5,0.6]$, we deduce from
 Proposition~\ref{prop:nex-lattice-precise} that
 \begin{align*}
  \lpr_1((0.5,0.6])&=\min \{\lpr_{\ldf}((0.5,0.6]),\lpr_{\udf}((0.5,0.6])\}\\&=\min\{\ldf(0.6)-\ldf(0.5),\udf(0.6)-\udf(0.5)\}=\min\{0.2,0.1\}=0.1.
 \end{align*}
 However, it follows from Proposition~\ref{prop:nex-lattice} that
 $\lnex_{\ldf,\udf}((0.5,0.6])=\max\{0,\ldf(0.6)-\udf(0.5)\}=\max\{0,0.2-0.5\}=0$.
\end{example}

\subsection{Inner Measure}

The inner measure $\lpboxlattice_*$ of the coherent lower
probability $\lpboxlattice$ defined in Eq.~\eqref{eq:nex-lattice}
coincides with $\lnepbox$ on all events \cite[Cor.~3.1.9,
p.~127]{walley1991}:
\begin{equation}\label{eq:inner-measure}
  \lnepbox(A)=\lpboxlattice_*(A)=\sup_{C\in\plattice,C\subseteq A}\lpboxlattice(C),
\end{equation}

The next example demonstrates that the choice of the preorder $\preceq$ can have a significant impact, even for the same $\ldf$ and $\udf$.

\begin{example}
\label{ex:ordering-importance}
Take $\pspace=\{0,1,2,3,4\}$ and consider:
\begin{equation*}
\ldf(x)=\udf(x)=
\begin{cases}
  0 & \text{if } x<2, \\
  1 & \text{otherwise}.
\end{cases}
\end{equation*}
Consider the total preorder $\preceq_1$ defined by $0 \simeq_1 1 \prec_1 2 \simeq_1 3 \simeq_1 4$ and the usual total ordering $\preceq_2$ defined by $0 \prec_2 1 \prec_2 2 \prec_2 3 \prec_2 4$. With $\preceq_1$, we have for any event $A\subseteq\pspace$ that
\begin{equation*}
  \lnepbox(A)=
  \begin{cases}
    1 & \text{if } \{2,3,4\}\subseteq A, \\
    0 & \text{otherwise}.
  \end{cases}
\end{equation*}
using Eqs.~\eqref{eq:nex-lattice} and~\eqref{eq:inner-measure}. However, with $\preceq_2$,
\begin{equation*}
  \lnepbox(A)=
  \begin{cases}
    1 & \text{if } 2\in A, \\
    0 & \text{otherwise}.
  \end{cases}
\end{equation*}
\end{example}

For ease of notation, from now onwards, we denote $\lnepbox$ by
$\lnex$ when no confusion about the cumulative distribution functions determining the p-box can arise.

In principle, the problem of natural extension to all events is
solved: simply calculate the inner measure as in
Eq.~\eqref{eq:inner-measure}, using Eq.~\eqref{eq:nex-lattice} to
calculate $\lpboxlattice(C)$ for elements $C$ in $\plattice$. However,
the inner measure still involves calculating a supremum, which may be
non-obvious. What we show next is that Eq.~\eqref{eq:nex-lattice} can
be extended to arbitrary events, by first taking the topological
interior with respect to a very simple topology, followed by a
(possibly infinite) sum over the so-called full components of this
interior.

\subsection{The Partition Topology}

Consider the \emph{partition topology} on $\pspace$ generated by the equivalence classes $[x]_\simeq$, that is, the topology generated by
\begin{equation*}
  \tau:=\{[x]_\simeq: x\in \pspace\}.
\end{equation*}
This topology is very similar to the discrete topology, except that it
is not Hausdorff, unless $\preceq$ is anti-symmetric: if $x\simeq y$
but $x\neq y$ then $x$ and $y$ cannot be topologically separated,
since every neighborhood of $x$ is also a neighborhood of $y$ and vice
versa. If $\preceq$ is anti-symmetric (for example, the usual ordering
on the reals is), then the partition topology reduces to the discrete
topology, that is, every set is clopen (closed and open).

The open sets in this topology are all unions of
equivalence classes (or, subsets of $\pspace/\simeq$, if you
like). Hence, in this topology, every open set is also closed. In
particular, every interval in $(\pspace,\preceq)$ is clopen.

The topological interior of a set $A$ is given by the union of all
equivalence classes contained in $A$:
\begin{equation}\label{eq:interior}
  \interior{A}
  =\bigcup\{[x]_\simeq\colon [x]_\simeq\subseteq A\}
\end{equation}
and the topological closure is given by the union of all equivalence
classes which intersect with $A$:
\begin{equation}\label{eq:closure}
  \closure{A}
  =\bigcup\{[x]_\simeq\colon [x]_\simeq\cap A\neq\emptyset\}.
\end{equation}

\begin{lemma}\label{lem:lnex-interior}
For any subset $A$ of $\pspace$, $\lnex(A)=\lnex(\interior{A})$ and $\unex(A)=\unex(\closure{A})$.
\end{lemma}
\begin{proof}
 Clearly $\lnex(\interior{A})\le\lnex(A)$ because $\interior{A}\subseteq A$. If we can also
 show that $\lnex(\interior{A})\ge\lnex(A)$ the desired result is established.

 Consider an element $C$ of the field $\plattice$ which is included in
 $A$. Since $C$ is in particular an open set in the partition topology,
 it is a subset of $\interior{A}$. The monotonicity of $\lnex$ implies that
 $\lnex(C)\leq \lnex(\interior{A})$. Consequently,
 \begin{equation*}
   \lnex(A)
   =\lpboxlattice_*(A)
   =\sup_{C\in\plattice,C\subseteq A}\lpboxlattice(C)
   =\sup_{C\in\plattice,C\subseteq A}\lnex(C)
   \le\lnex(\interior{A}).
 \end{equation*}
 The equality $\unex(A)=\unex(\closure{A})$ now follows from the fact that $\closure{A}=(\interior{A^c})^c$ and $\unex(A)=1-\lnex(A^c)$
\end{proof}

\subsection{Additivity on Full Components}

Next, we determine a constructive expression of the natural extension
$\lnex$ on the clopen subsets of $\pspace$.

\begin{definition}\cite[\S 4.4]{schechter1997}\label{def:full}
A set $S\subseteq\pspace$ is called \emph{full} if $[a,b]\subseteq S$ for any $a\preceq
b$ in $S$.
\end{definition}

What do these full sets look like? Obviously, any full set is clopen,
as it must be a union of equivalence classes.

\begin{lemma}\label{lem:full-is-clopen}
  Every full set is clopen.
\end{lemma}
\begin{proof}
  Observe that $[a,a]=[a]_\simeq$, and apply Definition~\ref{def:full}.
\end{proof}

Under an additional completeness assumption, the full sets are
precisely the intervals.

\begin{lemma}\label{lem:order-complete-intervals}
  If $\pspace/\simeq$ is order
  complete, that is, if every subset of $\pspace/\simeq$ has a supremum
  (minimal upper bound) and infimum (maximal lower bound), then every
  full set is an interval, that is, it can be written as $[x,y]$,
  $[x,y)$, $(x,y]$, or $(x,y)$, for some $x$, $y$ in
$\pspace$.
\end{lemma}
\begin{proof}
  Consider a full set $S$ in $\pspace$. Since, by Lemma~\ref{lem:full-is-clopen}, $S$ is clopen,
  we may consider $S$ as a subset of
  $\pspace/\simeq$.  So $[x]_{\simeq}=\inf S$ and $[y]_{\simeq}=\sup
  S$ exist, by the order completeness of $\pspace/\simeq$. Consider
  the case $x\in S$ and $y\not\in S$. Apply the definitions of $\inf$
  and $\sup$ to establish that $z\in S$ if and only if $x\preceq
  z\prec y$. The other three cases are proven similarly.
\end{proof}

Note that $\pspace/\simeq$ can be made order complete via the Dedekind
completion \cite[\S 4.34]{schechter1997}.

\begin{definition}\cite[\S 4.4]{schechter1997}\label{def:components}
Given a clopen set $A\subseteq\pspace$ and an element $x$ of $A$, the \emph{full
component} $C(x,A)$ of $x$ in $A$ is the largest full set $S$ which
satisfies $x\in S\subseteq A$.
\end{definition}

\begin{lemma}\label{lem:full-components-partition}
  The full components of any clopen set $A$ form a partition of $A$.
\end{lemma}
\begin{proof}
  This is easily shown using a similar result for total orders given
  in \cite[4.4(a)]{schechter1997}.
\end{proof}

In the following theorem, we prove that the natural extension $\lnex$
is additive on full components. Recall that the sum of a (possibly
infinite) family $(x_\lambda)_{\lambda\in\Lambda}$ of non-negative
real numbers is defined as
\begin{equation*}
  \sum_{\lambda\in\Lambda}x_\lambda=\sup_{\substack{L\subseteq\Lambda\\L\text{ finite}}}\sum_{\lambda\in L}x_\lambda
\end{equation*}
If the outcome of the above sum is a finite number, at most
countably many of the $x_\lambda$'s are non-zero
\cite[10.40]{schechter1997}.

\begin{theorem}\label{thm:component-additivity}
  Let $B$ be a clopen subset of $\pspace$. Let
  $(B_\lambda)_{\lambda\in\Lambda}$ be the full components of $B$, and
  let $(C_\lambda)_{\lambda\in\Lambda'}$ be the full components of
  $B^c$. Then
  \begin{align*}
    \lnex(B)&=\sum_{\lambda\in\Lambda}\lnex(B_\lambda)
    \\
    \unex(B)&=1-\sum_{\lambda\in\Lambda'}\lnex(C_\lambda)
  \end{align*}
\end{theorem}
\begin{proof}
 Since $(B_\lambda)_{\lambda\in\Lambda}$ is a partition of $B$, and because the functional $\lnex$ is monotone and super-additive,
  \begin{equation*}
    \lnex(B)
    =\lnex(\cup_{\lambda\in\Lambda}B_\lambda)
    \ge
    \lnex(\cup_{\lambda\in L} B_\lambda)
    \ge
    \sum_{\lambda\in L}\lnex(B_\lambda)
  \end{equation*}
  for every finite subset $L$ of $\Lambda$, and consequently
  $\lnex(B)\ge\sum_{\lambda\in\Lambda}\lnex(B_\lambda)$.
  We are left to show that also
  $\lnex(B)\le\sum_{\lambda\in\Lambda}\lnex(B_\lambda)$.

  Let $\epsilon>0$. By Eq.~\eqref{eq:inner-measure}, the inner measure coincides with the natural extension, so there is a $C$ in $\plattice$ such that
  $C\subseteq B$ and
  $\lnex(B)\le\lnex(C)+\epsilon$. From Eq.~\eqref{eq:plattice}, we can write
  \begin{equation*}
  C=(x_0,x_1]\cup(x_2,x_3]\cup\dots\cup(x_{2n},x_{2n+1}]
  \end{equation*}
  for some $x_0\prec x_1\prec \dots\prec x_{2n+1}$ in $\pspace^*$.
  Let us denote $C_k:=(x_{2k},x_{2k+1}]$ for $k=0$, \dots, $n$. It
  is easily established that $C_0$, \dots, $C_n$ are the full components of $C$.

  Applying Proposition~\ref{prop:nex-lattice} twice, and using the fact that each full component $C_k$ of $C$ must be a subset of exactly one full component $B_\lambda$ of $B$ (this is a consequence of $C\subseteq B$), we find that
  \begin{equation*}
    \lnex(C)
    =\sum_{k=0}^n \lnex(C_k)
    =\sum_{\lambda\in\Lambda}\lnex(\cup_{k\colon C_k\subseteq B_\lambda}C_k)
  \end{equation*}
  (note that the union is $\emptyset$ for those $\lambda$ where no $C_k\subseteq B_\lambda$---so only a finite number of terms can be non-zero in the latter sum) and consequently
  \begin{equation*}
    \lnex(B)
    \le\lnex(C)+\epsilon
    =\sum_{\lambda\in\Lambda}\lnex(\cup_{k\colon C_k\subseteq B_\lambda}C_k)+\epsilon
    \le\sum_{\lambda\in\Lambda}\lnex(B_\lambda)+\epsilon,
  \end{equation*}
  taking into account the monotonicity
  of $\lnex$. As this holds for any $\epsilon>0$, we arrive at the desired inequality.

  The expression for $\unex$ simply follows from the conjugacy relation $\unex(B)=1-\lnex(B^c)$, once noted that $B^c$ is clopen as well.
\end{proof}

In other words, the natural extension $\lnex$ of a p-box is \emph{arbitrarily
additive on full components}. In particular, interestingly, it is
$\sigma$-additive on full components (but obviously not additive, let
alone $\sigma$-additive, on arbitrary events).

\begin{example}
  Additivity on full components is not sufficient for a lower
  probability to be equivalent to a p-box, even in the finite
  case. For example, consider $\pspace=\{1,2,3\}$ with the usual
  ordering, so $\plattice=\wp(\pspace)$. Let $\lpr$ be the lower probability
  defined by
  \begin{equation*}
    \lpr(\{1\})=\lpr(\{2\})=\lpr(\{3\})=0.1
  \end{equation*}
  It can be checked that $\lpr$ is coherent, and that the natural extension
  $\lnex$ of $\lpr$ is the lower envelope of the probability mass functions
  $(0.8,0.1,0.1)$, $(0.1,0.8,0.1)$, and
  $(0.1,0.1,0.8)$. Moreover, $\lnex$ is additive on full components, because
  \begin{equation*}
    \lnex(\{1,3\})=\lnex(\{1\})+\lnex(\{3\})
  \end{equation*}
  (every other subset of $\pspace$ is already full). However,
  \begin{equation*}
    \lnex(\{2\})\neq\max\{0,\lnex(\{1,2\})-\unex(\{1\})\}
  \end{equation*}
  because $\lnex(\{1,2\})=0.2$ and $\unex(\{1\})=0.8$. This shows that $\lnex$
  violates Proposition~\ref{prop:nex-lattice} and as a consequence it is not the natural extension to events of a
  p-box.
\end{example}

\subsection{Summary}

Let us summarize all results so far, and explain how, in practice, $\lnex(A)$ and $\unex(A)$ of an arbitrary event $A$ can be calculated.

Proposition~\ref{prop:nex-lattice} gave the natural extension to the field
$\plattice$; we are now in a position to generalize it to all events,
at least when $\pspace/\simeq$ is order complete.

Indeed, consider an arbitrary event $A$. By Lemma~\ref{lem:lnex-interior}, it
suffices to calculate the natural extension of $\interior{A}$ or $\closure{A}$.
Calculating the interior or closure with respect to the partition topology will usually be trivial (see examples further on), and
moreover, the
topological interior or closure of a set is always clopen, so now we only need to know
the natural extension of clopen sets.

Now, by Theorem~\ref{thm:component-additivity}, it
follows that we only need to calculate the natural extension of the
full components $(B_\lambda)_{\lambda\in\Lambda}$ of
$\interior{A}$ or the full components $(C_\lambda)_{\lambda\in\Lambda}$ of $\closure{A}^c=\interior{A^c}$---note that each of these full components
is also clopen. Also, finding the full components will often be a trivial operation---commonly, there will only be a few.

But, by Lemma~\ref{lem:order-complete-intervals}, if, in addition,
$\pspace/\simeq$ is order complete, then each full component is an
interval. And for intervals, we immediately infer from
Proposition~\ref{prop:nex-lattice} and Eq.~\eqref{eq:inner-measure}
that:
\begin{subequations}\label{eq:nex-intervals}
\begin{align}
\lnex((x,y])&=\max\{0,\ldf(y)-\udf(x)\}
\\
\lnex((x,y))&=\max\{0,\ldf(y-)-\udf(x)\}
\\
\lnex([x,y])&=
\begin{cases}
  \max\{0,\ldf(y)-\udf(x)\} & \text{ if $x$ has no immediate predecessor}
  \\
  \max\{0,\ldf(y)-\udf(x-)\} & \text{ if $x$ has an immediate predecessor}
\end{cases}
\\
\lnex([x,y))&=
\begin{cases}
  \max\{0,\ldf(y-)-\udf(x)\} & \text{ if $x$ has no immediate predecessor}
  \\
  \max\{0,\ldf(y-)-\udf(x-)\} & \text{ if $x$ has an immediate predecessor}
\end{cases}
\end{align}
\end{subequations}
for any $x\prec y$ in $\pspace$,\footnote{In case $x=0_\pspace$,
evidently, $0_\pspace-$ is the immediate predecessor.} where
$\ldf(y-)$ denotes $\sup_{z\prec y}\ldf(z)$ and similarly for
$\udf(x-)$. The equalities hold because, if $x\prec y$ in $\pspace$,
and $x-$ is an immediate predecessor of $x$, then $[x,y]=(x-,y]$
and $[x,y)=(x-,y)$. Recall also that
$\ldf(0_\pspace-)=\udf(0_\pspace-)=0$ by convention. If
$\pspace/\simeq$ is finite, then one can think of $z-$ as the
immediate predecessor of $z$ in the quotient space $\pspace/\simeq$
for any $z \in \pspace$;

In other words, we have a simple constructive means of calculating the
natural extension of any event.

\subsection{Special Cases}

The above equations hold for any $(\pspace,\preceq)$ with order
complete quotient space. In most cases in practice, either
\begin{itemize}
\item $\pspace/\simeq$ is finite, or
\item $\pspace/\simeq$ is connected
  meaning that for any two elements $x\prec y$ in $\pspace$ there is a
  $z$ in $\pspace$ such that $x\prec z\prec y$,\footnote{This
    terminology stems from the fact that, in this case,
    $\pspace/\simeq$ is connected with respect to the order topology
    \cite[\S 15.46(6)]{schechter1997}.} (this is the case for instance when $\pspace$ is a closed
  interval in $\reals$ and $\preceq$ is the usual ordering of reals)
\end{itemize}
Moreover, if $\pspace/\simeq$ is connected, then $\ldf$ will
commonly satisfy $\ldf(y-)=\ldf(y)$ for all $y$ in $\pspace$. For
example, in case $\pspace$ is a closed interval in $\reals$, this
happens precisely when $\ldf(0)=0$ and $\ldf$ is left-continuous in
the usual sense.

Obviously, if $\pspace/\simeq$ is finite, then every element of
$\pspace$ has an immediate predecessor (remember, we take the immediate predecessor of
$0_\pspace$ to be $0_\pspace-$), and if $\pspace/\simeq$ is connected,
then no element except $0_\pspace$ has an immediate predecessor.

By Lemma~\ref{lem:lnex-interior}, Theorem~\ref{thm:component-additivity}, and Eqs.~\eqref{eq:nex-intervals}, we conclude:

\begin{corollary}\label{cor:nex-finite}
  If $\pspace/\simeq$ is finite, then every full set $B$ is of the
  form $[a,b]$, and
  \begin{align*}
    \lnex(A)&=\sum_{\lambda\in\Lambda}\max\{0,\ldf(b_\lambda)-\udf(a_\lambda-)\}
    \\
    \unex(A)&=1-\sum_{\lambda\in\Lambda'}\max\{0,\ldf(b'_\lambda)-\udf(a'_\lambda-)\}
  \end{align*}
  where $([a_\lambda,b_\lambda])_{\lambda\in\Lambda}$ are the full
  components of $\interior{A}$, and
  $([a'_\lambda,b'_\lambda])_{\lambda\in\Lambda'}$ are the full
  components of $\interior{A^c}=\closure{A}^c$.
\end{corollary}

\begin{corollary}\label{cor:nex-connected}
  If $\pspace/\simeq$ is order complete and connected, and
  $\ldf(y-)=\ldf(y)$ for all $y$ in $\pspace$, then
  \begin{align*}
    \lnex(A)&=\sum_{\lambda\in\Lambda}\max\{0,\ldf(\sup B_\lambda)-\udf(\inf B_\lambda)\}
    \\
    \unex(A)&=1-\sum_{\lambda\in\Lambda'}\max\{0,\ldf(\sup C_\lambda)-\udf(\inf C_\lambda)\}
  \end{align*}
  where $(B_\lambda)_{\lambda\in\Lambda}$ are the full components of
  $\interior{A}$
  and $(C_\lambda)_{\lambda\in\Lambda'}$ are the full components of
  $\interior{A^c}=\closure{A}^c$.
\end{corollary}

Beware of $\ldf(0_\pspace)=\ldf(0_\pspace-)=0$ in the last corollary.

\subsection{Example}
\label{sec:example-diagonal}

Let's investigate a particular type of p-boxes on the unit square
$[0,1]^2$. First, we must specify a preorder on $\pspace$. A natural
yet naive way of doing so is, for instance, saying that
$(x_1,y_1)\preceq (x_2,y_2)$ whenever
\begin{equation*}
  x_1+y_1\le x_2+y_2
\end{equation*}
Consider a p-box $(\ldf,\udf)$ on $([0,1]^2,\preceq)$. Since $\ldf$ is
required to be non-decreasing with respect to $\preceq$, it follows
that $\ldf(x,y)$ is constant on elements of $[0,1]^2/\simeq$, which
means that $F(x_1,y_1)=F(x_2,y_2)$ whenever $x_1+y_1=x_2+y_2$. Thus,
we may think of $\ldf(x,y)$ as a function of a single variable
$z=x+y$, and we write $\ldf(z)$. Similarly, we write $\udf(z)$.

Our definition of $\preceq$ means that our p-box specifies bounds on
the probability of right-angled triangles (restricted to $[0,1]^2$) whose hypothenuses are orthogonal to the diagonal:
\begin{equation}
\label{eq:diag-prob-bounds}
  \ldf(z)\le p(\{(x,y)\in[0,1]^2\colon x+y\le z\})\le\udf(z)
\end{equation}
Observe that the p-box is given directly on the two-dimensional product space, without the need to define marginal p-boxes for each dimension.
The base $\tau$ for our partition topology is given by
\begin{equation*}
  \tau=\{\{(x,y)\in[0,1]^2\colon x+y=z\}\colon z\in[0,2]\}
\end{equation*}
For example, the topological interior of a rectangle
$A=[a,b]\times[c,d]$ is empty, unless $a=c=0$ or $b=d=1$, because in
all other cases, no element of $\tau$ is a subset of $A$. In the
cases where $a=c=0$ and $\min\{b,d\}<1$, or $\max\{a,c\}>0$ and
$b=d=1$ (if $a=c=0$ and $b=d=1$ then the interior is $\pspace$),
respectively, we have:
\begin{align*}
  \interior{[0,b]\times[0,d]}&=\{(x,y) \in[0,1]^2 \colon x+y\le\min\{b,d\}\} \\
  \interior{[a,1]\times[c,1]}&=\{(x,y) \in[0,1]^2 \colon x+y\ge1+\max\{a,c\}\}
\end{align*}
Consequently, $\lnex(A)=0$ for all rectangles $A$, except for
\begin{align*}
  \lnex([0,b]\times[0,d])&=\ldf(\min\{b,d\}) \\
  \lnex([a,1]\times[c,1])&=1-\udf(1+ \max\{a,c\})
\end{align*}
Fig.~\ref{fig:example-diagonal} illustrates the situation. So, for the purpose of making inferences about the lower probability
of events that are rectangles, the ordering $\preceq$ was obviously
poorly chosen. In general, \emph{one should choose $\preceq$ in a way
  that $\pspace/\simeq$ contains good approximations for all events of
  interest.}

  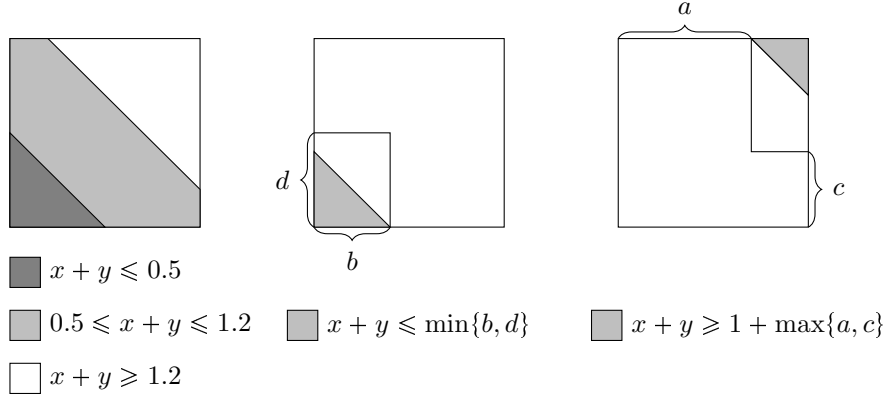
\begin{figure}
  \begin{tikzpicture}
  \begin{scope}[scale=2.5]
  \draw (0,0) rectangle (1,1);
  \draw[fill=gray!50] (0.2,1) -- (0,1) -- (0,0) -- (1,0) -- (1,0.2) -- cycle;
  \draw[fill=gray] (0,0.5) -- (0.5,0) -- (0,0) -- cycle;
   \draw[fill=gray,scale=0.8] (0,-0.4) -- (0.2,-0.4) -- (0.2,-0.2) node[midway,right] {$x+y \leq 0.5$} -- (0,-0.2) -- cycle;
  \draw[fill=gray!50,scale=0.8,yshift=-10] (0,-0.4) -- (0.2,-0.4) -- (0.2,-0.2) node[midway,right] {$0.5\leq x+y \leq 1.2$} -- (0,-0.2) -- cycle;
  \draw[scale=0.8,yshift=-20] (0,-0.4) -- (0.2,-0.4) -- (0.2,-0.2) node[midway,right] {$x+y \geq 1.2$} -- (0,-0.2) -- cycle;
  \end{scope}
   \begin{scope}[scale=2.5,xshift=1.6cm]
  \draw (0,0) rectangle (1,1);
  \draw (0,0.5) -- (0.4,0.5) -- (0.4,0);
  \draw[decoration={brace,amplitude=5pt},decorate] (0.4,0) -- (0,0) node[midway,below,yshift=-0.2cm] {$b$};
  \draw[decoration={brace,amplitude=5pt},decorate] (0,0) -- (0,0.5) node[midway,left,xshift=-0.2cm] {$d$};
  \draw[fill=gray!50] (0.4,0) -- (0,0.4) -- (0,0) -- cycle;
  \draw[fill=gray!50,scale=0.8,yshift=-10,xshift=-5] (0,-0.4) -- (0.2,-0.4) -- (0.2,-0.2) node[midway,right] {$x+y \leq \min\{b,d\}$} -- (0,-0.2) -- cycle;
  \end{scope}
     \begin{scope}[scale=2.5,xshift=3.2cm]
  \draw (0,0) rectangle (1,1);
  \draw (0.7,1) -- (0.7,0.4) -- (1,0.4);
  \draw (0.7,1) -- (1,0.7);
  \draw[decoration={brace,amplitude=5pt},decorate] (0,1) -- (0.7,1) node[midway,above,yshift=0.2cm] {$a$};
  \draw[decoration={brace,mirror,amplitude=5pt},decorate] (1,0) -- (1,0.4) node[midway,right,xshift=0.2cm] {$c$};
  \draw[fill=gray!50] (0.7,1) -- (1,0.7) -- (1,1) -- cycle;
    \draw[fill=gray!50,scale=0.8,yshift=-10,xshift=-5] (0,-0.4) -- (0.2,-0.4) -- (0.2,-0.2) node[midway,right] {$x+y \geq 1+ \max\{a,c\}$} -- (0,-0.2) -- cycle;
  \end{scope}
  \end{tikzpicture}
  \caption{Shape of intervals induced by $\preceq$, and calculation of the topological interior.}
  \label{fig:example-diagonal}
  \end{figure}

For example, in the case of rectangles, we could for instance discretize
$[0,1]^2$ into smaller squares, and impose some ordering on these
squares. Of course, it may not be entirely obvious how to interpret
the lower and upper cumulative distribution functions on such
discretized space, since there is no natural ordering on such
discretization. Another strategy would be to start from a reference point (e.g., an elicited modal value) and then to choose the ordering $\preceq$ such that intervals correspond to concentric regions of interests around the reference point.
Again, all of this is possible because our theory concerns p-boxes on arbitrary totally preordered spaces, and is not limited to the real line with its natural ordering.
More realistic examples in which such concentric regions are used are given in Section~\ref{sec:example:engineering}.

\section{Natural Extension to All Gambles}
\label{sec:nex-gambles}

Next, we establish that p-boxes are completely monotone, and that
therefore their natural extension to all gambles can be expressed as a
Choquet integral. We further simplify the calculation of this Choquet
integral via the lower and upper oscillation of gambles with respect to the
partition topology introduced earlier.

\subsection{Complete Monotonicity}

As shown in \cite[Section~3.1]{miranda2006c}, the natural extension
$\lnex_\df$ of a distribution function $\df$ on $[0,1]$ is
completely monotone. It is fairly easy to generalise this result to
distribution functions on an arbitrary totally preordered space
$\pspace$.\footnote{Indeed, by \cite[Thm.~5 \&\ Thm.~9]{cooman2008c},
  the natural extension of any finitely additive probability on a
  field is completely monotone.}
However, given this, and Eq.~\eqref{eq:lower-envelope-df}, we cannot
immediately deduce the complete monotonicity of $\lnex_{\ldf,\udf}$,
because the lower envelope of a set of completely monotone lower
previsions is not necessarily completely monotone. We prove next
that such an envelope is indeed completely monotone in the case of
p-boxes. This is an improvement with respect to previous
results \cite{desterckedubois2008}, where the relation between
p-boxes and complete monotonicity was established for
finite spaces.

Let $\lpboxlattice$ denote the restriction of $\lnepbox$ to
$\plattice$, given by Proposition~\ref{prop:nex-lattice}.

\begin{theorem}
  \label{th:pboxmon}
  $\lpboxlattice$ is a completely monotone coherent lower probability.
\end{theorem}
\begin{proof}
  Clearly, $\lpboxlattice$ is coherent as it is the natural extension
  to $\plattice$ of a coherent lower probability $\lpbox$
  \cite[p.~123, 3.1.2(a)]{walley1991}. To prove that it is
  completely monotone, we must establish that for all $p\in\SetN$, $2\leq p\leq n$,
  and all $A_1$, \dots, $A_p$ in $\plattice$:
  \begin{equation}\label{eq:complete:monotonicity}
    \lpboxlattice\left(\bigcup_{i=1}^pA_i\right)
    \geq\sum_{\emptyset\neq I\subseteq\{1,\dots,p\}}(-1)^{|I|+1}
    \lpboxlattice\left(\bigcap_{i\in I}A_i\right).
  \end{equation}

  For any $p\in\SetN$, $2\leq p\leq n$, and any $A_1$, \dots, $A_p$ in
  $\plattice$, consider the finite field generated by $A_1$,
  \dots, $A_p$. Let $\alpr$ denote the restriction of $\lpboxlattice$
  to this finite field. By
  \cite[Sec.~3]{desterckedubois2008}, $\alpr$ is completely monotone
  on this finite field.
  In particular, Eq.~\eqref{eq:complete:monotonicity} is
  satisfied. But this means that Eq.~\eqref{eq:complete:monotonicity} is
  satisfied for all $p\in\SetN$, $2\leq p\leq n$, and all $A_1$,
  \dots, $A_p$ in $\plattice$, which establishes the theorem.
\end{proof}

\subsection{Choquet Integral Representation}

Complete monotonicity allows us to characterise the natural extension on all
gambles, as we show in the following theorem:

\begin{theorem}\label{theo:natex-choquet}
  The natural extension $\lnex$ of
  $\lpbox$ is given by the Choquet integral
  \begin{equation*}
    \lnex(f)
    =\inf f + \int_{\inf f}^{\sup f} \lnex(\{f \geq t\}) \dif t
  \end{equation*}
  for every gamble $f$. Moreover, $\lnex$ is completely monotone on all gambles. Similarly,
  \begin{equation*}
    \unex(f)
    =\inf f + \int_{\inf f}^{\sup f} \unex(\{f \geq t\}) \dif t.
  \end{equation*}
\end{theorem}
\begin{proof}
  Immediate from Theorem~\ref{th:pboxmon} once observed that $\lnex=\lpboxlattice_*$ on all events, and \cite[Theorems~8 and~9]{cooman2008c}. The latter two theorems state that:
  \begin{itemize}
  \item Given a coherent $n$-monotone ($n \geq 2$) lower probability $\lpr$ defined on a field (here, $\plattice$), its natural extension to all gambles is given by its Choquet integral $(C)\int\cdot\dif\lpr_*$
  \item If a coherent lower probability $\lpr$ defined on a field is $n$-monotone ($n \geq 2$), then its natural extension to all gambles is $n$-monotone.
  \end{itemize}
\end{proof}

\subsection{Lower and Upper Oscillation}

By Lemma~\ref{lem:lnex-interior}, to turn
Theorem~\ref{theo:natex-choquet} in an effective algorithm, we must
calculate $\interior{\{f\ge t\}}$ for every $t$. Fortunately, there is
a very simple way to do this.

For any
gamble $f$ on $\pspace$ and any topological base $\tau$, define its \emph{lower oscillation}
as the gamble
\begin{align}
  \nonumber
  \losc(f)(x)&=\sup_{C \in \tau\colon x\in C} \inf_{y \in C} f(y)
  \\
  \intertext{For the partition topology which we introduced earlier, this simplifies to}
  \label{eq:losc}
 \losc(f)(x) &=\inf_{y \in [x]_\simeq} f(y)
\end{align}
The upper oscillation is:
\begin{equation}
  \label{eq:uosc}
  \uosc(f)(x)=-\losc(-f)(x)=\sup_{y \in [x]_\simeq} f(y)
\end{equation}
For a subset $A$ of $\pspace$, we deduce from the above definition
and from Eq.~\eqref{eq:interior} that the lower oscillation of $I_A$
is $I_{\interior{A}}$, so the lower oscillation is the natural
generalisation of the topological interior to gambles. Similarly, we
see from Eq.~\eqref{eq:closure} that the upper oscillation of $I_A$
is $I_{\closure{A}}$.

\begin{proposition}\label{prop:lnex-losc}
  For any gamble $f$ on $\pspace$,
  \begin{subequations}\label{eq:prop:lnex-losc:int:cl}
  \begin{align}
    \interior{\{f\ge t\}}&=\{\losc(f)\ge t\}
    \\
    \closure{\{f\ge t\}}&=\{\uosc(f)\ge t\}
  \end{align}
  \end{subequations}
  so, in particular,
  \begin{subequations}\label{eq:prop:lnex-losc:choquet}
  \begin{align}
    \lnex(f)
    &=
    \inf\losc(f)+ \int_{\inf\losc(f)}^{\sup\losc(f)}
    \lnex(\{\losc(f) \geq t\}) \dif t
    =\lnex(\losc(f))
    \\
    \unex(f)
    &=
    \inf\uosc(f) + \int_{\inf\uosc(f)}^{\sup\uosc(f)}
    \unex(\{\uosc(f) \geq t\}) \dif t
    =\unex(\uosc(f))
  \end{align}
  \end{subequations}
\end{proposition}
\begin{proof}
  Eqs.~\eqref{eq:prop:lnex-losc:int:cl} are easily established using
  the definitions of interior and closure, and lower and upper
  oscillation. For example,
  \begin{equation*}
    x\in\interior{\{f\ge t\}}
  \end{equation*}
  if and only if there is a $C$ in $\tau$
  such that $x\in C$ and
  \begin{equation*}
    C\subseteq \{f\ge t\}
  \end{equation*}
  But, for our choice of $\tau$, necessarily $C=[x]_\simeq$ if $x\in C\in \tau$, so the above holds if and only if
  \begin{equation*}
    \forall y\in [x]_\simeq\colon f(y)\ge t
  \end{equation*}
  And this holds if and only if
  \begin{equation*}
    \losc(f)(x)\ge t
  \end{equation*}
  where we used the defintion of $\losc(f)$.

  Let us prove Eqs.~\eqref{eq:prop:lnex-losc:choquet}. It follows from
  Eq.~\eqref{eq:losc} that $f\geq\losc(f)$, and as a consequence
  $\lnex(f)\geq \lnex(\losc(f))$. We are left to prove that
  $\lnex(f)\leq \lnex(\losc(f))$.

  Indeed, using Lemma~\ref{lem:lnex-interior}, and
  Eqs.~\eqref{eq:prop:lnex-losc:int:cl},
  \begin{align*}
    \lnex(f)
    &=\inf f + \int_{\inf f}^{\sup f} \lnex(\{f \geq t\})\dif t
    =\inf f + \int_{\inf f}^{\sup f} \lnex(\{\losc(f) \geq t\}) \dif t
    \\
    \intertext{and since obviously, by Eq.~\eqref{eq:losc}, $\inf f=\inf\losc(f)$,}
    &=\inf\losc(f) + \int_{\inf\losc(f)}^{\sup f} \lnex(\{\losc(f) \geq t\}) \dif t
    \\
    \intertext{and since $\sup f\ge \sup\losc(f)$, using the usual properties of the Riemann integral,}
    &=\inf\losc(f) + \int_{\inf\losc(f)}^{\sup\losc(f)} \lnex(\{\losc(f) \geq t\}) \dif t
    \\
    &\quad
    + \int_{\sup\losc(f)}^{\sup f} \lnex(\{\losc(f) \geq t\}) \dif t
  \end{align*}
  Now use the fact that $\{\losc(f) \geq t\}=\emptyset$ for
  $t\in(\sup \ \losc(f),\sup f]$, so the last term is zero.
\end{proof}

Concluding, to calculate the natural extension of any gamble, in
practice, we must simply determine the full components of the cut sets
of its lower or upper oscillation, and calculate a simple Riemann
integral of a monotonic function.

Examples will be given in Section~\ref{sec:example:engineering}.

\section{P-Boxes Whose Preorders are Induced by a Real-Valued Function}
\label{sec:induced-preorder}

In practice, the most convenient way to specify a preorder $\preceq$
on $\pspace$ such that $\pspace/\simeq$ is order complete and
connected is by means of a bounded real-valued function
$Z\colon\pspace\to\reals$. For instance, in the example in
Section~\ref{sec:example-diagonal}, we used $Z(x,y)=x+y$. Also see
\cite{2006:benhaim}
and \cite{2008:fuchs}.

Let us assume from now onwards that $Z$ is a surjective mapping from
$\pspace$ to $[0,1]$.

For any $x$ and $y$ in $\pspace$, define $x\preceq y$ whenever
$Z(x)\le Z(y)$. Because $Z$ is surjective, $\pspace/\simeq$ is order
complete and connected. In particular, $\pspace$ has a smallest and largest element,
for which $Z(0_\pspace)=0$ and $Z(1_\pspace)=1$. Moreover, we can
think of any cumulative distribution function on $(\pspace,\preceq)$
as a function over a single variable $z\in[0,1]$. Consequently, we can
think of any p-box on $(\pspace,\preceq)$ as a p-box on
$([0,1],\le)$. In particular, for any subset $I$ of $[0,1]$ we
write $\lnex(I)$ for $\lnex(Z^{-1}(I))$. For example, for $a$, $b$ in
$[0,1]$, and $A=Z^{-1}((a,b])\subseteq\pspace$, we have that
\begin{equation*}
  \lnex(A)=\lnex((a,b])=\max\{0,\ldf(a)-\udf(b)\}
\end{equation*}
by Proposition~\ref{prop:nex-lattice}. Similar expressions for other types
of intervals follow from Eqs.~\eqref{eq:nex-intervals}.

The topological interior and closure can be related to the so-called
\emph{lower and upper inverse} of $Z^{-1}$. Indeed, consider
the multi-valued mapping $\Gamma:=Z^{-1}\colon[0,1]\rightarrow
\wp(\pspace)$. Because for every $x$ in $\pspace$, it holds that
$[x]_\simeq=\Gamma(Z(x))$, it follows that,
for any subset $A$ of $\pspace$,
$\interior{A}=\Gamma(\Gamma_*(A))$, and
$\closure{A}=\Gamma(\Gamma^*(A))$, where $\Gamma_*$ and $\Gamma^*$
denote the lower and upper inverse of $\Gamma$ respectively, that
is
\begin{align*}
  \Gamma_*(A)&=\{z\in[0,1]\colon \Gamma(z)\subseteq A\}, \text{ and}
  \\
  \Gamma^*(A)&=\{z\in[0,1]\colon \Gamma(z)\cap A\neq\emptyset\}
\end{align*}
(see for instance \cite{dempster1967}).\footnote{We follow the
terminology in~\cite{nguyen1978b} and~\cite{DuboisPrade1987}, among
others. Beware that $\Gamma_*$ and $\Gamma^*$ are sometimes called
\emph{upper} and \emph{lower} inverse instead
\cite{1997:tsiporkova:upperinverse,1959:berge}, or \emph{strong}
and \emph{weak} inverse \cite{1965:debreu}.}

\begin{theorem}\label{thm:nex-z}
  Let $A$ be an arbitrary subset of $\pspace$. Then
  \begin{align*}
    \lnex(A)&=\sum_{\lambda\in\Lambda}\lnex(I_\lambda)
    \\
    \unex(A)&=1-\sum_{\lambda\in\Lambda'}\lnex(J_\lambda)
  \end{align*}
  where $(I_\lambda)_{\lambda\in\Lambda}$ are the full components of
  $Z(\interior{A})=\Gamma_*(A)$
  and $(J_\lambda)_{\lambda\in\Lambda'}$ are the full components of
  $Z(\interior{A^c})=Z(\closure{A}^c)=\Gamma_*(A^c)$.

  If, in addition, $\ldf$ is left-continuous as a function of
  $z\in[0,1]$ and $\ldf(0)=0$, then
  \begin{align*}
    \lnex(A)&=\sum_{\lambda\in\Lambda}\max\{0,\ldf(\sup I_\lambda)-\udf(\inf I_\lambda)\}
    \\
    \unex(A)&=1-\sum_{\lambda\in\Lambda'}\max\{0,\ldf(\sup J_\lambda)-\udf(\inf J_\lambda)\}
  \end{align*}
\end{theorem}
\begin{proof}
  Indeed, by Corollary~\ref{cor:nex-connected},
  \begin{equation*}
    \lnex(A)
    =
    \lnex(\interior{A})
    =
    \sum_{\lambda\in\Lambda}\lnex(B_\lambda)
  \end{equation*}
  where $(B_\lambda)_{\lambda\in\Lambda}$ are the full components of
  $\interior{A}$. So, the result is established if we can show that
  $(Z^{-1}(I_\lambda))_{\lambda\in\Lambda}$ are the full components of
  $\interior{A}$.

  Obviously, since $(I_\lambda)_{\lambda\in\Lambda}$ partitions
  $Z(\interior{A})$, it follows that
  $(Z^{-1}(I_\lambda))_{\lambda\in\Lambda}$ partitions
  \begin{equation*}
    \bigcup_{\lambda\in\Lambda}Z^{-1}(I_\lambda)=Z^{-1}(Z(\interior{A}))=\interior{A}
  \end{equation*}
  where the latter equality follows from the fact that $\interior{A}$
  is clopen, i.e., is a union of equivalence classes.

  We are left to prove that each set $Z^{-1}(I_\lambda)$ is a full
  component. Clearly, $Z^{-1}(I_\lambda)$ is full: for any two $x$ and
  $y$ in $Z^{-1}(I_\lambda)$, it holds that
  \begin{equation*}
    [x,y]
    =\{v\in\pspace\colon Z(x)\le Z(v)\le Z(y)\}
    =Z^{-1}([Z(x),Z(y)])
    \subseteq Z^{-1}(I_\lambda)
  \end{equation*}
  where we used the fact that $[Z(x),Z(y)]\subseteq I_\lambda$ in the
  last step.

  Consider any $x\in Z^{-1}(I_\lambda)$. The desired result is
  established if we can show that $Z^{-1}(I_\lambda)$ is the largest
  full set $S$ which satisfies $x\in S\subseteq \interior{A}$.

  Suppose $S$ is larger, that is, $S$ is full, $x\in S\subseteq
  \interior{A}$, and $Z^{-1}(I_\lambda)\subsetneq S$. Since both sets
  are clopen, it must be that there is some $y\in S$ such that
  $[y]_\simeq\cap Z^{-1}(I_\lambda)=\emptyset$. But this implies
  that
  \begin{equation*}
    I_\lambda=Z(Z^{-1}(I_\lambda))\subsetneq Z(S)
  \end{equation*}
  because $Z(y)$ belongs to $Z(S)$ but not to
  $Z(Z^{-1}(I_\lambda))$. But, this would mean that $I_\lambda$ is not
  a full component of $Z(\interior{A})$---a contradiction. So,
  $Z^{-1}(I_\lambda)$ must be the largest full set $S$ which satisfies
  $x\in S\subseteq \interior{A}$.
\end{proof}

Regarding gambles, note that the lower oscillation is constant on
equivalence classes (this follows immediately from its definition).
Hence, we may consider $\losc(f)$ for a gamble $f$ on $\pspace$ as a
function of $z\in[0,1]$, and we can use
Proposition~\ref{prop:lnex-losc} to write:
\begin{proposition}\label{prop:lnex-z-losc}
For any gamble $f$ on $\pspace$,
\begin{align*}
  \lnex(f)
  &=\inf \losc(f) + \int_{\inf \losc(f)}^{\sup \losc(f)} \lnex(\{z\colon\losc(f)(z) \geq t\}) \dif t
  \\
  \unex(f)
  &=\inf \uosc(f) + \int_{\inf \uosc(f)}^{\sup \uosc(f)} \unex(\{z\colon\uosc(f)(z) \geq t\}) \dif t
\end{align*}
\end{proposition}

\section{Constructing Multivariate P-Boxes from Marginals}
\label{sec:marginals}

In this section, we construct a
multivariate p-box from marginal coherent lower previsions under arbitrary rules of combination. As special cases, we derive expressions for the joint,
\begin{enumerate}[(i)]
\item either without any assumptions about
dependence or independence between variables, that is, using the
Fr\'echet-Hoeffding bounds \cite{hoeffding1963},
\item or assuming epistemic independence between all variables, that is,
  using the factorization property \cite{2010:decooman:ine}.
\end{enumerate}
We also derive Williamson and Downs's~\cite{WilliamsonDowns1990} probabilistic arithmetic as a special case of our framework.

Specifically, consider $n$ variables $X_1$, \dots, $X_n$ assuming values in $\mathcal{X}_1$, \dots, $\mathcal{X}_n$, and assume that marginal
lower previsions $\lpr_1$, \dots, $\lpr_n$, are given for each
variable---each of these could be the natural extension of a p-box, although we do not require this. So, each $\lpr_i$ is a coherent lower prevision on $\gambles[\mathcal{X}_i]$.

\subsection{Multivariate P-Boxes}

The first step in constructing our multivariate p-box is to define a mapping $Z$ to
induce a preorder $\preceq$ on $\pspace=\mathcal{X}_1\times\dots\times \mathcal{X}_n$. The following choice
works perfectly for our purpose:
\begin{equation*}
  Z(x_1,\dots,x_n)=\max_{i=1}^n Z_i(x_i)
\end{equation*}
where each $Z_i$ is a surjective mapping from $\mathcal{X}_i$ to $[0,1]$ and
hence, also induces an marginal preorder $\preceq_i$ on $\mathcal{X}_i$. Each
$\lpr_i$ can be approximated by a p-box $(\ldf_i,\udf_i)$ on
$(\mathcal{X}_i,\preceq_i)$, defined by
\begin{align*}
  \ldf_i(z)&=\lpr_i(Z_i^{-1}([0,z])) & \udf_i(z)&=\upr_i(Z_i^{-1}([0,z]))
\end{align*}
This approximation is the best possible one, by
Theorem~\ref{thm:pbox-approximation}.

Beware that even though different choices of $Z_i$ may induce the same
total preorder $\preceq_i$, they might lead to a different total
preorder $\preceq$ induced by $Z$. Whence, our joint total preorder
$\preceq$ is not uniquely determined by $\preceq_i$. Roughly speaking,
the $Z_i$ specify how the marginal preorders $\preceq_i$ scale relative to one another.  As
we shall see, this effectively means that our choice of $Z_i$ affects
the precision of our inferences: a good choice will ensure that any
event of interest can be well approximated by elements of
$\pspace/\simeq$. Of course, nothing prevents us, at least in theory,
to consider the set of all $Z_i$ which induce some given marginal
total preorders $\preceq_i$, and whence to work with a set of
p-boxes. In some cases, this may result in quite complicated
calculations. However, in Section~\ref{sec:prob:arithm}, we will see
an example where this approach is feasible.

Anyway, with this choice of $Z$, we can easily
find the p-box which represents the joint as accurately as possible, under any rule of combination of coherent lower previsions:

\begin{lemma}\label{lem:joint-p-box}
  Consider any rule of combination $\odot$ of coherent lower and upper previsions, mapping the marginals $\lpr_1$, \dots, $\lpr_n$ to a joint coherent lower prevision $\bigodot_{i=1}^n\lpr_i$ on all gambles. Suppose there are functions $\ell$ and $u$ for which:
  \begin{align*}
  \bigodot_{i=1}^n\lpr_i\left(\prod_{i=1}^n A_i\right)
  &=
  \ell(\lpr_1(A_1),\dots,\lpr_n(A_n))\text{ and }
  \\
  \bigodot_{i=1}^n\upr_i\left(\prod_{i=1}^n A_i\right)
  &=
  u(\upr_1(A_1),\dots,\upr_n(A_n)),
  \end{align*}
  for all $A_1\subseteq \mathcal{X}_1$, \dots, $A_n\subseteq \mathcal{X}_n$. Then, the couple $(\ldf,\udf)$ defined by
  \begin{align*}
    \ldf(z)&=\ell(\ldf_1(z),\dots,\ldf_n(z)) &
    \udf(z)&=u(\udf_1(z),\dots,\udf_n(z))
  \end{align*}
  is the least conservative p-box on $(\pspace,\preceq)$ whose natural
  extension $\lnepbox$ is dominated by the combination
  $\bigodot_{i=1}^n\lpr_i$ of $\lpr_1$, \dots, $\lpr_n$.
\end{lemma}
\begin{proof}
  By Theorem~\ref{thm:pbox-approximation}, the least conservative
  p-box on $(\pspace,\preceq)$ whose natural extension is dominated by
  the joint $\lpr=\bigodot_{i=1}^n\lpr_i$ is given by
  \begin{align*}
    \ldf(z)&=\lpr(Z^{-1}([0,z]))
    &
    \udf(z)&=\upr(Z^{-1}([0,z]))
  \end{align*}
  Now, observe that the set $Z^{-1}([0,z])$ is a product of marginal intervals:
  \begin{align*}
    Z^{-1}([0,z])
    &=
    \{
    (x_1,\dots,x_n)\in\pspace\colon \max_{i=1}^n Z_i(x_i)\le z
    \}
    \\
    &=
    \{
    (x_1,\dots,x_n)\in\pspace\colon (\forall i=1,\dots,n)(Z_i(x_i)\le z)
    \}
    \\
    &=
    \prod_{i=1}^n\{x_i\in X_i\colon Z_i(x_i)\le z\}
    =
    \prod_{i=1}^nZ_i^{-1}([0,z]).
  \end{align*}
  The desired equalities follow immediately.
\end{proof}

\subsection{Natural Extension: The Fr\'echet Case}

The \emph{natural extension} $\boxtimes_{i=1}^n\lpr_i$ of $\lpr_1$,
\dots, $\lpr_n$ is the lower envelope of all joint distributions (or,
linear previsions) whose marginal distributions (or, marginal linear
previsions) are compatible with the given marginal lower
previsions. So, the model is completely vacuous about the dependence
structure, as it includes all possible forms of dependence. We refer
to for instance \cite[p.~120, \S 3.1]{cooman2003c} for a rigorous
definition. In this paper, we only need the following equalities,
which are known as the Fr\'echet bounds~\cite{1935:frechet} (see for instance~\cite[p.~131]{1989:williamson}
for a more recent discussion):
\begin{subequations}
  \label{eq:joint-natural-extension}
\begin{align}
  \bigboxtimes_{i=1}^n\lpr_i\left(\prod_{i=1}^n A_i\right)
  &=\max\left\{0,1-n+\sum_{i=1}^n\lpr_i(A_i)\right\}\text{ and }
  \\
  \bigboxtimes_{i=1}^n\upr_i\left(\prod_{i=1}^n A_i\right)
  &=\min_{i=1}^n\upr_i(A_i)
\end{align}
\end{subequations}
for all $A_1\subseteq \mathcal{X}_1$, \dots, $A_n\subseteq \mathcal{X}_n$.

\begin{theorem}\label{th:joint-alldep}
  The p-box $(\ldf,\udf)$ defined by
  \begin{align*}
    \ldf(z)&=\max\left\{0,1-n+\sum_{i=1}^n\ldf_i(z)\right\} &
    \udf(z)&=\min_{i=1}^n\udf_i(z)
  \end{align*}
  is the least conservative p-box on $(\pspace,\preceq)$ whose natural
  extension $\lnepbox$ is dominated by the natural extension
  $\boxtimes_{i=1}^n\lpr_i$ of $\lpr_1$, \dots, $\lpr_n$.
\end{theorem}
\begin{proof}
  Immediate, by Lemma~\ref{lem:joint-p-box} and
  Eqs.~\eqref{eq:joint-natural-extension}.
\end{proof}

The next example shows that, even when $\lpr_i$ are p-boxes, the joint
p-box will in general only be an outer approximation
(although the closest one that is a p-box) of the joint lower prevision.

\begin{example}\label{exm:joint-pbox-outer-approx}
Consider two variables $X$ and $Y$ with domain $\mathcal{X}=\{x_1,x_2\}$, with $x_1 \prec x_2$, and $\mathcal{Y}=\{y_1,y_2\}$, with $y_1 \prec y_2$. Consider
\begin{align*}
\ldf_1(x_1)&=0.4, & \udf_1(x_1)&=0.6, & \ldf_1(x_2)&=\udf_1(x_2)=1, \\
\ldf_2(y_1)&=0.2, & \udf_2(y_1)&=0.3, & \ldf_2(y_2)&=\udf_2(y_2)=1.
\end{align*}
Let $\lpr_1$ be the natural extension of $(\ldf_1,\udf_1)$, and let $\lpr_2$ be the natural extension of $(\ldf_2,\udf_2)$.
Consider the events $A=\{x_1\} \times \mathcal{Y}$ and $B=\mathcal{X} \times \{y_2\}$. Writing $\lpr$ for $\boxtimes_{i=1}^n\lpr_i$, we have that
\begin{align*}
  \lpr(A)&=\lpr_1(\{x_1\})=\max\{0,\ldf_1(x_1)-\udf_1(x_1-)\}=0.4
  \\
  \lpr(B)&=\lpr_2(\{y_2\})=\max\{0,\ldf_2(y_2)-\udf_2(y_1)\}=0.7
\end{align*}
whence,
\begin{align*}
  \lpr(A \cup B)
  &
  =\max\{\lpr(A),\lpr(B)\}
  =0.7,
  \\
  \lpr(A \cap B)
  &
  =\max\{0,1-2+\lpr(A)+\lpr(B)\}
  =0.1.
\end{align*}
But this means that $\lpr$ is not even 2-monotone, because $\lpr(A \cup B) + \lpr(A \cap B) < \lpr(A) + \lpr(B)$. Therefore, $\lpr$ cannot be represented by a p-box, as p-boxes are completely monotone by Theorem~\ref{th:pboxmon}.
\end{example}

\subsection{Independent Natural Extension}

In contrast, the \emph{independent natural extension}
$\otimes_{i=1}^n\lpr_i$ of $\lpr_1$, \dots, $\lpr_n$ models epistemic
independence between $X_1$, \dots, $X_n$. We refer to
\cite{2010:decooman:ine} for a rigorous definition and properties. In
this paper, we only need the following equalities:
\begin{subequations}
  \label{eq:joint-independent-natural-extension}
\begin{align}
  \bigotimes_{i=1}^n\lpr_i\left(\prod_{i=1}^n A_i\right)
  &=
  \prod_{i=1}^n\lpr_i(A_i)\text{ and }
  \\
  \bigotimes_{i=1}^n\upr_i\left(\prod_{i=1}^n A_i\right)
  &=
  \prod_{i=1}^n\upr_i(A_i)
\end{align}
\end{subequations}
for all $A_1\subseteq \mathcal{X}_1$, \dots, $A_n\subseteq \mathcal{X}_n$.

\begin{theorem}\label{th:joint-independent}
  The p-box $(\ldf,\udf)$ defined by
  \begin{align*}
    \ldf(z)&=\prod_{i=1}^n\ldf_i(z) &
    \udf(z)&=\prod_{i=1}^n\udf_i(z)
  \end{align*}
  is the least conservative p-box on $(\pspace,\preceq)$ whose natural
  extension $\lnepbox$ is dominated by the indepedent natural extension
  $\otimes_{i=1}^n\lpr_i$ of $\lpr_1$, \dots, $\lpr_n$.
\end{theorem}
\begin{proof}
  Immediate, by Lemma~\ref{lem:joint-p-box} and
  Eqs.~\eqref{eq:joint-independent-natural-extension}.
\end{proof}

Again, in general, the joint p-box will only be an outer approximation
of the actual joint lower prevision.

\begin{example}\label{exm:ine-joint-pbox-outer-approx}
Again, consider two variables $X$ and $Y$ with domain $\mathcal{X}=\{x_1,x_2\}$, with $x_1 \prec x_2$, and $\mathcal{Y}=\{y_1,y_2\}$, with $y_1 \prec y_2$. Consider
\begin{align*}
\ldf_1(x_1)&=0.4, & \udf_1(x_1)&=0.6, & \ldf_1(x_2)&=\udf_1(x_2)=1, \\
\ldf_2(y_1)&=0.3, & \udf_2(y_1)&=0.5, & \ldf_2(y_2)&=\udf_2(y_2)=1.
\end{align*}
As before, let $\lpr_1$ be the natural extension of $(\ldf_1,\udf_1)$, and let $\lpr_2$ be the natural extension of $(\ldf_2,\udf_2)$.
Consider the events $A=\{(x_1,y_1),(x_1,y_2)\}$ and $B=\{(x_1,y_2),(x_2,y_1)\}$. Writing $\lpr$ for $\otimes_{i=1}^n\lpr_i$, we have that
\begin{align*}
  \lpr(A)&=\lpr_1(\{x_1\})=0.4
  \\
  \lpr(B)
  &\ge
  0.4
\end{align*}
where the last inequality follows from the fact that all probability mass functions $p$ which dominate $\lpr$ must satisfy $p(x_1|y_2)\ge\lpr(\{x_1\})=0.4$ and $p(x_2|y_1)\ge\lpr(\{x_2\})=0.4$, whence
\begin{equation*}
  p(B)=p(x_1|y_2)p(y_2)+p(x_2|y_1)p(y_1)\ge 0.4(p(y_1)+p(y_2))=0.4
\end{equation*}
for all $p$ which dominate $\lpr$. Because $\lpr$ is the lower envelope of all such $p$, the desired inequality follows.\footnote{By
linear programming, it can actually be shown that $\lpr(B)=0.4$.}
Also, because of the factorization property of the independent natural extension,
\begin{align*}
  \lpr(A \cup B)
  &=1-\upr(\{(x_2,y_2)\})
  =1-\upr_1(\{x_2\})\upr_2(\{y_2\})=1-0.6\times 0.7
  =0.58
  \\
  \lpr(A \cap B)
  &=\lpr(\{(x_1,y_2)\})
  =\lpr_1(\{x_1\})\lpr_2(\{y_2\})
  =0.4\times 0.5=0.2.
\end{align*}
Again, this means that $\lpr$ cannot be represented by a p-box, as it violates 2-monotonicity.
\end{example}

\subsection{Special Case: Probabilistic Arithmetic}
\label{sec:prob:arithm}

Let $Y=X_1+X_2$ with $X_1$ and $X_2$ real-valued random variables. One
can also consider substraction, multiplication, and division, but for
simplicity, we stick to addition---the other three cases follow
along almost identical lines.

Probabilistic
arithmetic~\cite{1989:williamson} deals with the problem of estimating
$\lpr_Y([-\infty,y])=\ldf_Y(y)$ and $\upr_Y([-\infty,y])=\udf_Y(y)$ for
any $y \in \reals$ under the assumptions that
\begin{itemize} 
\item the uncertainty on $X_1$ and $X_2$ is given by p-boxes
  $(\ldf_1,\udf_1)$ and $(\ldf_2,\udf_2)$, with $\preceq_1$ and
  $\preceq_2$ the natural ordering of real numbers,\footnote{For
    substraction and division, $\preceq_2$ is the reverse natural
    ordering.} and
\item the dependence structure is completely
  unknown (Fr\'echet case).
\end{itemize}
Using the Fr\'echet bounds, Williamson and
Downs~\cite{WilliamsonDowns1990} provide explicit formulae for the
different arithmetic operations, thus providing very efficient
algorithms to make inferences from marginal p-boxes.

Let us show, for the particular case of addition,
that their results are captured by
our joint p-box proposed in Theorem~\ref{th:joint-alldep}.
Other
cases, not treated here to save space, follow from almost identical
reasoning.\footnote{Note that $X_1$ and $X_2$ are assumed to be positive
  in case of multiplication and division.}
The lower cumulative
distribution function resulting from probabilistic arithmetic is, for
any $y \in \reals$
\begin{equation}\label{eq:willdownsplus}
  \ldf_{X_1 + X_2}(y)=\sup_{x_1,x_2\colon x_1 + x_2=y}\max\{0,\ldf_{1}(x_1)+\ldf_2(x_2)-1\}.
\end{equation}
Without much loss of generality, and for our convenience, assume that
both $X_1$ and $X_2$ lie in a bounded interval $[a,b]$.

Let $Z_1$ and $Z_2$ be any surjective maps $[a,b]\to[0,1]$ which
induce the usual ordering on $[0,1]$. Some properties of $Z_1$ and
$Z_2$ immediately follow: both are continuous and strictly increasing,
and so are their inverses---we rely on this in a bit.

To apply Theorem~\ref{th:joint-alldep}, we
consider the total preorder $\preceq$ on $\pspace=[a,b]^2$ induced by
$Z(x_1,x_2)=\max\{Z_1(x_1),Z_2(x_2)\}$.
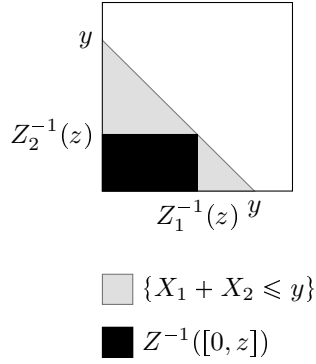
\begin{figure}
  \begin{tikzpicture}
  \begin{scope}[scale=2.5]
  \draw (0,0) rectangle (1,1) ;
  \draw[fill=gray!50,opacity=0.5] (0,0.8) node[left,opacity=1] {$y$} -- (0,0) -- (0.8,0) node[below,opacity=1] {$y$} -- cycle;
  \draw[fill=black] (0,0) -- (0,0.3) node[left]  {$Z_2^{-1}(z)$} -- (0.5,0.3) -- (0.5,0) node[below] {$Z_1^{-1}(z)$} -- cycle;

  \draw[fill=gray!50,opacity=0.5,scale=0.8,yshift=-10] (0,-0.4) -- (0.2,-0.4) -- (0.2,-0.2) node[midway,right,opacity=1] {$\{X_1 + X_2 \leq y\}$} -- (0,-0.2) -- cycle;
  \draw[scale=0.8,yshift=-20,fill=black] (0,-0.4) -- (0.2,-0.4) -- (0.2,-0.2) node[midway,right] {$Z^{-1}([0,z])$} -- (0,-0.2) -- cycle;
  \end{scope}
  \end{tikzpicture}
  \caption{The event $\{X_1 + X_2\le y\}$, and the largest interval $Z^{-1}([0,z])$ included in it.}
  \label{fig:WilliamDowns}
\end{figure}
Figure~\ref{fig:WilliamDowns} illustrates the event\footnote{
By $\{X_1 +X_2\le y\}$ we mean $\{(x_1,x_2)\in[0,1]^2\colon x_1+x_2\le y\}$.}
$\{X_1 +X_2\le y\}$, with $y\in[2a,2b]$, as well as the largest interval $Z^{-1}([0,z])$ included in it. Recall that
\begin{equation*}
  Z^{-1}([0,z])=Z_1^{-1}([0,z])\times Z_2^{-1}([0,z])=[0,Z_1^{-1}(z)]\times[0,Z_2^{-1}(z)].
\end{equation*}
Whence, for $z$ such that $Z_1^{-1}(z)+Z_2^{-1}(z)=y$, we achieve the
largest interval $Z^{-1}([0,z])$ which is still included in $\{X_1
+X_2\le y\}$. There is always a unique such $z$ because
also $Z_1^{-1}+Z_2^{-1}$ is continuous and strictly increasing.

Recall that, by Theorem~\ref{th:joint-alldep} (now without shortcuts
in notation),
\begin{align*}
  \ldf(Z^{-1}(z))&=\max\left\{0,\ldf_1(Z^{-1}_1(z))+\ldf_2(Z^{-1}_2(z))-1\right\}
  \\
  \udf(Z^{-1}(z))&=\min_{i=1}^n\{\udf_1(Z^{-1}_1(z)),\udf_2(Z^{-1}_2(z))\}
\end{align*}
is the least conservative p-box on $(\pspace,\preceq)$ whose natural
extension is dominated by the natural extension
$\lpr_1\boxtimes\lpr_2$ of $\lpr_1$ and $\lpr_2$.

Also, as we have just shown, $Z^{-1}([0,z])$, for our choice of $z$, is the
topological interior of $\{X_1 +X_2\le y\}$. Whence, by
Theorem~\ref{thm:nex-z}, we find that
\begin{align*}
  \lnepbox(\{X_1+X_2\le y\})
  &=
  \lnepbox(Z^{-1}([0,z]))
  =
  \ldf(Z^{-1}(z))
  \\
  &=
  \max\{0,\ldf_1(Z^{-1}_1(z))+\ldf_2(Z^{-1}_2(z))-1\}
\end{align*}
where we remember that $z$ is chosen such that
$Z_1^{-1}(z)+Z_2^{-1}(z)=y$.

But, this holds for every valid choice of functions $Z_1$ and $Z_2$, whence
\begin{equation*}
  \lpr_1\boxtimes\lpr_2(\{X_1+X_2\le y\})
  \ge
  \sup_{x_1,x_2\colon x_1+x_2=y}\max\{0,\ldf_1(x_1)+\ldf_2(x_2)-1\}
\end{equation*}
which indeed coincides with Eq.~\eqref{eq:willdownsplus}. Similar
arguments hold for the upper cumulative distribution functions, and
other arithmetic operations.

In conclusion, probabilistic arithmetic constitutes a very
specific case of our approach.

\section{Examples}
\label{sec:example:engineering}

In this section, we investigate two different examples in which p-boxes are used to model uncertainty around some parameters. The first example concerns a damped harmonic oscillator, i.e., a classical engineering toy example. The second example concerns the evaluation of a river dike height, an important issue in regions subject to potential floods.

\subsection{Damped Harmonic Oscillator}

Consider a simple damped harmonic oscillator, with damping coefficient
$c>0$, spring constant $k>0$, and mass $m>0$. The \emph{damping ratio}
\begin{equation*}
  \zeta(c,k)=\frac{c}{2\sqrt{km}}
\end{equation*}
determines how quickly the oscillator returns to its equilibrium
state---$\zeta(c,k)=1$ means fastest convergence. Suppose the engineering
design has already been completed, so the optimal values for $c^*$ and
$k^*$ have been determined, such that $\zeta(c^*,k^*)=1$.

Without loss of generality, we choose the units for mass, time, and
length, such that $m=k^*=\zeta(c^*,k^*)=1$ (so $c^*=2$).

Of course, the actual values for $c$ and $k$ will differ from their
design values, and uncertainty must be taken into account. Let us
calculate the lower and upper expectation of $\zeta(c,k)$, given that our
uncertainty about $c$ and $k$ is described by a p-box.

First, we must specify a preorder. For this problem, it seems fairly
natural to have bounds on the quantiles of the distance between the
actual values $(c,k)$ and the design values $(2,1)$. This comes
down to for instance the following choice for $Z$:
\begin{equation*}
  Z(c,k)=\max\{|c-2|,2|k-1|\}
\end{equation*}
For simplicity, we only consider the region $Z(c,k)\le 1$. This means that we are certain that $c\in[1,3]$ and $k\in[0.5,1.5]$---if
necessary, $Z$ can be rescaled to accomodate larger or smaller
regions. Note that
we have taken a supremum norm as distance. This simplifies the
calculations below, but of course, one might as well take the
Euclidian norm, or any other reasonable distance function, at the
expense of slightly more complicated calculations and dependency modelling (see Fig.~\ref{fig:harmonic:equivclasses}).

\begin{figure}
\begin{tikzpicture}
\tikzstyle{blackdot}=[circle,draw=black,fill=black,scale=0.4]
\draw[->] (0,0) -- (4cm,0) node[below] {$c$};
\draw[->] (0,0) -- (0,4cm) node[left] {$k$};
\draw (2cm,0.1cm) -- (2cm,-0.1cm) node[below] {$2$};
\draw (0.1cm,2cm) -- (-0.1cm,2cm) node[left] {$1$};
\draw[thick] (1cm,1cm) rectangle (3cm,3cm);
\node at (2cm,2cm) {$[(c,k)]_\simeq$};
\node[blackdot] at (1cm,3cm) { };
\node[blackdot] at (3cm,1cm) {};
\node[above] at (1cm,3cm) { $\losc(\zeta)(z)$};
\node[below] at (3cm,1cm) {$\uosc(\zeta)(z)$};
\node at (2cm,-1cm) {Rectangular equivalence class.};

\begin{scope}[xshift=5.5cm]
\draw[->] (0,0) -- (4cm,0) node[below] {$c$};
\draw[->] (0,0) -- (0,4cm) node[left] {$k$};
\draw (2cm,0.1) -- (2cm,-0.1cm) node[below] {$2$};
\draw (0.1cm,2cm) -- (-0.1cm,2cm) node[left] {$1$};
\begin{scope}[xshift=2cm,yshift=2cm]
\draw[thick,rotate=-45] (0,0) ellipse (0.6cm and 2cm);
\node at (0cm,0cm) {$[(c,k)]_\simeq$};
\end{scope}
\node at (2cm,-1cm) {Elliptical equivalence class.};
\end{scope}
\end{tikzpicture}
\caption{Different possible equivalence classes.}
\label{fig:harmonic:equivclasses}
\end{figure}
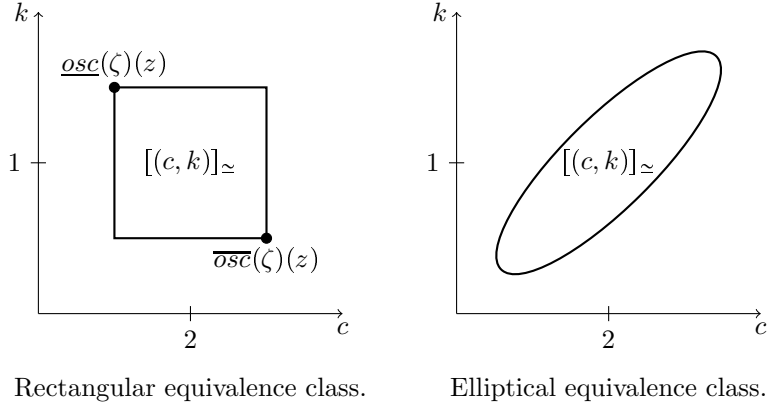

Equivalence classes $[(c,k)]_\simeq$ are edges of rectangles with
vertices $$(2\pm Z(c,k),1\pm Z(c,k)/2).$$

What is a p-box for the preorder $\preceq$ induced by $Z$? A p-box
$(\ldf,\udf)$ specifies lower and upper bounds for the probability of
concentric rectangles around the design point $(2,1)$:
\begin{equation*}
  \ldf(z)\le p(\{(c,k)\colon Z(c,k)\le z\})\le \udf(z)
\end{equation*}
So, effectively, our p-box specifies concentric prediction regions for
the uncertain parameters $c$ and $k$.

We can now calculate the lower and upper expectation of
$\zeta(c,k)$. First, we calculate the lower oscillation
$\losc(\zeta)$ and upper oscillation $\uosc(\zeta)$ (see Fig.~\ref{fig:losc:uosc}):
\begin{align*}
  \losc(\zeta)(z)
  &=\inf_{(c,k)\colon Z(c,k)=z}\zeta(c,k)
  =\frac{2-z}{2\sqrt{(1+z/2)}}
  \\
  \uosc(\zeta)(z)
  &=\sup_{(c,k)\colon Z(c,k)=z}\zeta(c,k)
  =\frac{2+z}{2\sqrt{(1-z/2)}}
\end{align*}
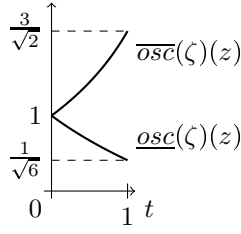
\begin{figure}
\begin{tikzpicture}
\draw[->] (-0.1,0) -- (1.1,0) node[below right] {$t$};
\draw[->] (0,-0.1) -- (0,0) node[below left] {$0$} -- (0,0.4082) node[left]{$\tfrac{1}{\sqrt{6}}$} -- (0,1) node[left] {$1$} -- (0,2.1213) node[left] {$\tfrac{3}{\sqrt{2}}$} -- (0,2.5);
\draw[domain=0:1, thick]  plot[id=losc] function{(2-x)/(2*((1+x/2)**0.5))};
\draw[domain=0:1, thick]  plot[id=uosc] function{(2+x)/(2*((1-x/2)**0.5))};
\draw[dashed] (0, 0.4082) -- (1, 0.4082) node[above right] {$\losc(\zeta)(z)$};
\draw[dashed] (0, 2.1213) -- (1, 2.1213) node[below right] {$\uosc(\zeta)(z)$};
\draw (1,-0.1) node[below] {$1$} -- (1,0.1);
\end{tikzpicture}
\caption{The lower oscillation $\losc(\zeta)(z)$ and upper oscillation $\uosc(\zeta)(z)$.}
\label{fig:losc:uosc}
\end{figure}
Next, we find the full components of the events
\begin{align*}
  L_t
  &=\{z\in[0,1]\colon\losc(\zeta)(z)\ge t\}
  =
  \left\{z\in[0,1]\colon\frac{2-z}{2\sqrt{(1+z/2)}}\ge t\right\}
  \\
  U_t
  &=\{z\in[0,1]\colon\uosc(\zeta)(z)\ge t\}
  =
  \left\{z\in[0,1]\colon\frac{2+z}{2\sqrt{(1-z/2)}}\ge t\right\}
\end{align*}
for all $t\in[0,0.5]$. Fortunately, $\losc(\zeta)$ is decreasing as
function of $z$, and $\uosc(\zeta)$ is increasing, and hence
$L_t=[0,\ell_t]$ and $U_t=[u_t,1]$, with (see Fig.~\ref{fig:zoft})
\begin{align*}
  \ell_t&=2-t(-t+\sqrt{t^2+8}):=z(t) && \text{ with }t\in[\tfrac{1}{\sqrt{6}},1]
  \\
  u_t&=-2+t(-t+\sqrt{t^2+8}):=-z(t) && \text{ with }t\in[1,\tfrac{3}{\sqrt{2}}]
\end{align*}
Note that the given bounds for $t$ arise from the minimum and maximum
of $\losc(f)$ and $\uosc(f)$.  Concluding, by
Proposition~\ref{prop:lnex-z-losc}, when $\ldf(z)=\ldf(z-)$ for all $z\in[0,1]$ and $\ldf(0)=0$,
\begin{align*}
  \lnex(\zeta)
  &=
  \frac{1}{\sqrt{6}}+\int_{\frac{1}{\sqrt{6}}}^{1}
  \ldf(z(t))\dif t
  \\
  \unex(\zeta)
  &=
  1+\int_{1}^{3/\sqrt{2}}
  \big(1-\ldf(-z(t))\big)\dif t
\end{align*}
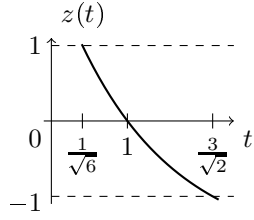
\begin{figure}
\begin{tikzpicture}
\draw[->] (-0.1,0) -- (2.4,0) node[below right] {$t$};
\draw[->] (0,-1.1) node[left] {$-1$} -- (0,0) node[below left] {$0$} -- (0,1) node[left] {$1$} -- (0,1.1) node[above right]{$z(t)$};
\draw[domain=0.4:2.2, thick]  plot[id=zoft] function{2-x*(-x+(x**2+8)**0.5)};
\draw[dashed] (0, 1) -- (2.4, 1);
\draw[dashed] (0,-1) -- (2.4,-1);
\draw (1,-0.1) node[below] {$1$} -- (1,0.1);
\draw (0.4082,-0.1) node[below] {$\tfrac{1}{\sqrt{6}}$} -- (0.4082,0.1);
\draw (2.1213,-0.1) node[below] {$\tfrac{3}{\sqrt{2}}$} -- (2.1213,0.1);
\end{tikzpicture}
\caption{The function $z(t)$ which determines the cut sets.}
\label{fig:zoft}
\end{figure}

Interestingly, both the lower and upper expectation of $\zeta$ are
determined by the lower cumulative distribution function only. Hence, in this problem, we
actually do not need to elicit the upper cumulative distribution
function.\footnote{Of course this will not always be the case---it
  just happens to be so for this example.}

For example, if the expert says that $c$ and $k$ are independent, and the marginal lower cumulative distribution functions are uniform on $c\in [1,2]$ and $k\in[0.5,1.5]$, so $\ldf_1(z)=\ldf_2(z)=z$, with preorders induced by $Z_1(c)=|c-2|$ and $Z_2(k)=2|k-1|$,
then, by Theorem~\ref{th:joint-independent}, because $Z=\max\{Z_1,Z_2\}$, it follows that $\ldf(z)=z^2$, and
\begin{align*}
  \lnex(\zeta)
  &=
  \frac{1}{\sqrt{6}}+\int_{\frac{1}{\sqrt{6}}}^{1}
  z(t)^2\dif t
  =0.584
  \\
  \unex(\zeta)
  &=
  1+\int_{1}^{3/\sqrt{2}}
  \big(1-z(t)^2\big)\dif t
  =1.664
\end{align*}

\subsection{River Dike Height Estimation}

We aim to estimate the minimal required dike
height along a given stretch of river, using a simplified model that is used
by the EDF (the French integrated energy operator) to make initial evaluations \cite{2008:rocquigny}. Although this model is quite simple, it provides a realistic
industrial application. Skipping technical details, the model results in the following relationship:
\begin{equation}
  \label{eq:EDFmodel}
  h(q,k,u,d)
  =
  \begin{cases}
    \left(\frac{q}{k\sqrt{\frac{u-d}{\ell}}b}\right)^{\frac{3}{5}} & \text{if }q\ge 0 \\
    0& \text{otherwise}.
  \end{cases}
\end{equation}
The meaning of the variables is summarised in Table~\ref{tab:EDFvar}.
\begin{table}
\begin{center}
\begin{tabular}{cll}
symbol & name & unit  \\
\hline
$h$ & overflow height of the river & $m$ \\
$q$ & river flow rate & $m^3 s^{-1}$\\
    $b$ & river width & $m$ \\
     $k$ & Strickler coefficient & $m^{1/3} s^{-1}$ \\
     $u$ & upriver water level & $m$\\
    $d$ & downriver water level & $m$\\
     $\ell$  & length of river stretch & $m$
\end{tabular}
\end{center}
\caption{Meaning of the variables used in Eq.~\eqref{eq:EDFmodel}}
\label{tab:EDFvar}
\end{table}

For the particular case under study, the river width is $b=300 m$ and the river length is $\ell=6400 m$. The remaining parameters are uncertain. Expert assessment leads to the following distributions:
\begin{itemize}
\item The maximal flow rate $q$ has a Gumbel distribution\footnote{The Gumbel distribution models the maximum of an exponentially distributed sample, and is used in extreme value theory \cite{2006:dehaan} to model rare events such as floods.} with
  location parameter $\mu=1335 m^3 s^{-1}$ and scale parameter $\beta=716
  m^3 s^{-1}$. For calculations, it is easier to work with
  symmetric distributions. Therefore, we introduce a variable $r$
  satisfying
  \begin{equation*}
    q=\mu-\beta\ln(-\ln(r)).
  \end{equation*}
  If $r$ is uniformly distributed over $[0,1]$, then $q$ has the
  Gumbel distribution with location parameter $\mu$ and scale
  parameter $\beta$.
  So, after transformation,
  \begin{equation*}
    h(r,k,u,d)
    =
    \begin{cases}
      \left(\frac{\mu-\beta\ln(-\ln(r))}{k\sqrt{\frac{u-d}{\ell}}b}\right)^{\frac{3}{5}} & \text{if $\mu-\beta\ln(-\ln(r))\ge 0$} \\
      0 & \text{otherwise}.
    \end{cases}
  \end{equation*}
\item The Strickler coefficient $k$ has a symmetric triangular
  distribution over the interval $[15 m^{1/3} s^{-1},45 m^{1/3} s^{-1}]$ (with mode
  at $k^*=30 m^{1/3} s^{-1}$).
\item There is also uncertainty about the water levels $u$ and $d$,
  because sedimentary conditions are hard to characterise. Measured
  values are $u^*=55 m$ and $d^*=50m$, with measurement error
  definitely less than $1m$. These are also modelled by symmetric
  triangular distributions, on $[54m,56m]$ and $[49m,51m]$
  respectively.
\end{itemize}

Again, a natural choice for $Z$ is the distance between the expected values $(r^*=1/2,k^*=30,u^*=55,d^*=50)$ and the actual values $(r,k,u,d)$:
\begin{equation*}
  Z(r,k,u,d)=\max\{2|r-1/2|,|k-30|/15,|u-55|,|d-50|\},
\end{equation*}
The scale of the distances has been chosen such that $Z(r,k,u,d) \leq 1$ for all points of interest. Equivalence classes $[(r,k,u,d)]_\simeq$ are borders of $4$-dimentional boxes with
vertices $$((1\pm z)/2, 30\pm 15 z, 55 \pm z,50\pm z)$$
where $z=Z(r,k,u,d)$.

The marginal p-boxes are, for $r$:
\begin{equation*}
  \ldf_1(z)=\udf_1(z)=p(2|r-1/2|\le z)=p(r\in[(1-z)/2,(1+z)/2])=z
\end{equation*}
because $r$ is uniformly distributed over $[0,1]$.
For $k$, we have:
\begin{equation*}
  \ldf_2(z)=\udf_2(z)=p(|k-30|/15\le z)=p(k\in[30-15 z,30+ 15 z])=1-(1-z)^2
\end{equation*}
(see Fig.~\ref{fig:pofkforldf}).
Similarly, for $u$ and $d$, it is easily verified that:
\begin{equation*}
  \ldf_3(z)=\udf_3(z)=\ldf_4(z)=\udf_4(z)=1-(1-z)^2
\end{equation*}

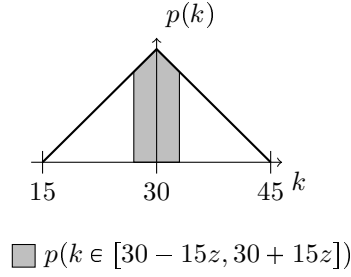
\begin{figure}
  \begin{tikzpicture}[scale=1.5]
    \draw[fill=gray!50] (-0.2,0) -- (-0.2,0.8) -- (0,1) -- (0.2,0.8) -- (0.2,0) -- cycle;
    \draw[fill=gray!50,yshift=-15,xshift=-36] (0,-0.4) -- (0.2,-0.4) -- (0.2,-0.2) node[midway,right] {$p(k\in[30-15 z,30+ 15 z])$} -- (0,-0.2) -- cycle;
    \draw[->] (-1.1,0) -- (1.1,0) node[below right] {$k$};
    \draw[->] (0,-0.1) node[below]{$30$} -- (0,1.1) node[above right]{$p(k)$};
    \draw[-,thick] (-1,0) -- (0,1) -- (1,0);
    \draw[-] (-1,-0.1) node[below]{$15$} -- (-1,0.1);
    \draw[-] (1,-0.1) node[below]{$45$} -- (1,0.1);
  \end{tikzpicture}
  \caption{Derivation of the p-box for a triangular distribution.}
  \label{fig:pofkforldf}
\end{figure}

The lower oscillation
$\losc(h)$ and upper oscillation $\uosc(h)$ can be calculated along the same lines as in the previous example:
\begin{align*}
  \losc(h)(z)
  &=\inf_{(r,k,u,d)\colon Z(r,k,u,d)=z}h(r,k,u,d)=o(-z)
  \\
  \uosc(h)(z)
  &=\sup_{(r,k,u,d)\colon Z(r,k,u,d)=z}h(r,k,u,d)
  =o(z)
\end{align*}
with
\begin{align*}
  o(z)&=
  \begin{cases}
    \left(\frac{\mu-\beta\ln(-\ln((1+z)/2))}{(30-15z)\sqrt{\frac{5-2z}{\ell}}b}\right)^{\frac{3}{5}} & \text{if $\mu-\beta\ln(-\ln((1+z)/2))\ge 0$}
    \\
    0 & \text{otherwise}.
  \end{cases}
\end{align*}
The function $o(z)$ is depicted in Fig.~\ref{fig:oofz}: it is increasing, with $o(-1)=0$ (this is not immediately clear from the picture, but at higher scale, it becomes apparent), $o(0)=3.032$, and $o(1)=+\infty$.
\begin{figure}
\begin{tikzpicture}[xscale=1, yscale=0.2]
\draw[->] (-1.1,0) -- (1.1,0) node[below right] {$z$};
\draw[->] (0,-0.5) node[below]{$0$} -- (0,11) node[above right]{$o(z)$};
\draw[-,dashed] (-1,0) -- (-1,11);
\draw[-,dashed] (1,0) -- (1,11);
\draw[-] (-1,-0.5) node[below]{$-1$} -- (-1,0.5);
\draw[-] (1,-0.5) node[below]{$1$} -- (1,0.5);
\draw[-] (-0.1,5) node[left]{$5$} -- (0.1,5);
\draw[-] (-0.1,10)  node[left]{$10$} -- (0.1,10);
\draw[smooth,domain=-0.99:0.99, thick]  plot[id=oofz,samples=100] function{(((1335.0-716.0*log(-log((1.0+x)/2.0))))/((30.0-15.0*x)*sqrt((5.0-2.0*x)/6400.0)*300.0))**(3.0/5.0)};
\end{tikzpicture}
\caption{The function $o(z)$ which determines the lower and upper oscillation, and the cut sets.}
\label{fig:oofz}
\end{figure}
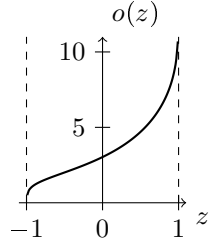

Again, $\losc(h)(z)$ and $\uosc(h)(z)$ are decreasing and increasing in $z$, respectively. Hence the full components of the events
\begin{align*}
  L_t
  &=\{z\in[0,1]\colon\losc(h)(z)\ge t\}=\{z\in[0,1]\colon o(-z)\ge t\}
  \\
  U_t
  &=\{z\in[0,1]\colon\uosc(h)(z)\ge t\}=\{z\in[0,1]\colon o(z)\ge t\}
\end{align*}
are of the form $L_t=[0,\ell_t]$ and $U_t=[u_t,1]$ again, with
\begin{align*}
  \ell_t&=-o^{-1}(t)\text{ for }t\le o(0)
  &
  u_t&=o^{-1}(t)\text{ for }t\ge o(0)
\end{align*}

As in the previous example, we do not need to elicit the upper cumulative distributions, and only the lower ones need to be given. With unknown dependence, using Theorem~\ref{th:joint-alldep}, we have
$$\ldf(z)=\max\{0, -3 + z + 3(1-(1-z)^2)\}$$
and whence
\begin{align*}
  \lnex(h)
  &=
  \int_{0}^{o(0)}
  \ldf(-o^{-1}(t))\dif t
  =
  1.515
  \\
  \unex(h)
  &=
  o(0)+\int_{o(0)}^{+\infty}
  \big(1-\ldf(o^{-1}(t))\big)\dif t
  =
  6.423
\end{align*}
Therefore, to be on the safe
side, we should consider average overflowing heights of
at least $6.5m$. For comparison, using traditional methods instead of p-boxes, $h$ has expectation $3.2m$, assuming independence between all variables---this lies between the lower and upper expectation that we just calculated, as expected. The interval is obviously much wider:
\begin{itemize}
\item because we have reduced a multivariate problem to a univariate one, whence, leading to imprecision due to the difference between lower and upper oscillation,
\item and because we have not made any assumption of independence, whence, leading to imprecision due to weaker assumptions.
\end{itemize}

Realistically, the decision maker may desire a dike height $t$ such that the upper probability of disaster $\unex(\{h\ge t\})$ is less than a given threshold. It is easily verified that:
\begin{equation*}
  \unex(\{h\ge t\})=\unex(\{\uosc(h)\ge t\})=1-\ldf(o^{-1}(t))
\end{equation*}
For instance, for $\unex(\{h\ge t\})=0.01$, we need $t$ to be $10.725m$.  For comparison, using traditional methods instead of p-boxes, $t$ needs to be about $9m$, assuming independence between all variables.

In both examples, analytical calculations are relatively simple due to the monotonicity of the target function with respect to the uncertain variables. Of course, this may not be the case in general.

\section{Conclusions}\label{sec:conclusions}

P-boxes are one of the most interesting imprecise probability models
from an operational point of view, because they are simply characterised by
a lower and an upper cumulative distribution function.
In this paper, for the purpose of multivariate modelling, we
studied inferences (lower and upper expectations in particular) from p-boxes on arbitrary totally preordered spaces.
For this purpose, we represented p-boxes as
coherent lower previsions, and studied their natural extension.

We used an as general as possible model by considering p-boxes whose
lower and upper cumulative distribution functions are defined on a
totally preordered space. Thereby, we extended the theory of p-boxes
from finite to infinite sets, and from total orders to total
preorders. This allowed us unify
p-boxes on finite spaces and on intervals of reals numbers, and
to extend the theory to
the multivariate case.

One very interesting result of this paper is a practical
means of calculating the natural extension of a p-box in this general setting. We proved
that the natural extension of a p-box is arbitrarily additive on
full components of clopen sets with respect to the partition
topology induced by equivalence classes of the underlying preorder
(Theorem~\ref{thm:component-additivity},
Corollaries~\ref{cor:nex-finite} and~\ref{cor:nex-connected}). We
also proved that the natural extension is completely monotone, and
therefore has a Choquet integral representation
(Theorem~\ref{theo:natex-choquet}). Consequently, to calculate the
natural extension, we proved that it suffices to calculate the full
components of the cut sets of the lower oscillation, followed by a
simple Riemann integral (Proposition~\ref{prop:lnex-losc}).

As a special case, we studied p-boxes whose preorders are induced by a
real-valued mapping. Such p-boxes are particularly attractive,
as they allow to build or elicit a multivariate
uncertainty model at once. They correspond to lower and upper
probabilistic bounds given over nested regions that can take arbitrary
shapes.

Consequently,
we provided a new tool to combine marginal p-boxes into a joint p-box, under arbitrary rules of combination, thereby allowing any type of dependency modelling (Lemma~\ref{lem:joint-p-box}). As examples, we considered two extreme cases: assuming nothing about dependence (Theorem~\ref{th:joint-alldep}), and assuming epistemic independence (Theorem~\ref{th:joint-independent}). Similar formulas are easily derived for any other rule of combination.
Moreover, Williamson and Downs's~\cite{WilliamsonDowns1990} probabilistic arithmetic obtains as a special case of our approach.

We demonstrated our methodology
on inference about a damped harmonic oscillator, and on a river dike assessment,
showing that calculations are generally straightforward.

Of course, many open problems regarding p-boxes remain to be investigated.
For instance, even though there need not be any relation between the preorder and
the dependency model---because one can, in theory at least,
always construct a multivariate p-box from marginals for any dependency model and any
 preorder---some combinations obviously lead to more
imprecision than others. Our choice led to simple mathematical expressions, but is perhaps not the best one possible in terms of precision. Can the dependency model inform the
choice of preorder, to arrive at tighter bounds?
Also,
the connection of p-boxes with other uncertainty
models, such as possibility measures and clouds, deserves further investigation.

\section*{Acknowledgements}

Work partially supported by a doctoral grant from the IRSN. We are
particularly grateful to Enrique Miranda for the many very fruitful
discussions, extremely useful suggestions, and various contributions
to this paper. We also thank Gert de Cooman and Didier Dubois for
their help with a very early draft of this paper.
Finally, we thank both reviewers, whose constructive comments
helped improving the presentation of the paper substantially.

\bibliographystyle{plain}
\bibliography{references}
\end{document}